\documentclass[11pt,oneside,reqno]{amsart}

\usepackage{amssymb}
\usepackage{algpseudocode}
\usepackage{algorithm}
\usepackage{verbatim}
\usepackage{graphicx}
\usepackage{placeins}
\usepackage{listings}
\usepackage{xcolor}
\usepackage{subcaption}
\usepackage[noadjust]{cite}
\RequirePackage{dsfont} 
 \scrollmode
\allowdisplaybreaks
\usepackage{graphicx}
\usepackage{epstopdf}
\usepackage{hyperref}

\usepackage{amsmath}
\usepackage{tikz}

\makeatletter
\newcommand{\xleftrightarrow}[2][]{\ext@arrow 3359\leftrightarrowfill@{#1}{#2}}
\makeatother

\definecolor{codegreen}{rgb}{0,0.6,0}
\definecolor{codegray}{rgb}{0.5,0.5,0.5}
\definecolor{codepurple}{rgb}{0.58,0,0.82}
\definecolor{backcolour}{rgb}{0.95,0.95,0.92}

\lstdefinestyle{list_style}{
  backgroundcolor=\color{backcolour}, commentstyle=\color{codegreen},
  keywordstyle=\color{magenta},
  numberstyle=\tiny\color{codegray},
  stringstyle=\color{codepurple},
  basicstyle=\ttfamily\footnotesize,
  breakatwhitespace=false,         
  breaklines=true,                 
  captionpos=b,                    
  keepspaces=true,                 
  numbers=left,                    
  numbersep=5pt,                  
  showspaces=false,                
  showstringspaces=false,
  showtabs=false,                  
  tabsize=2
}

\newcommand{\xdasharrow}[2][->]{
\tikz[baseline=-\the\dimexpr\fontdimen22\textfont2\relax]{
\node[anchor=south,font=\scriptsize, inner ysep=1.5pt,outer xsep=2.2pt](x){#2};
\draw[shorten <=3.4pt,shorten >=3.4pt,dashed,#1](x.south west)--(x.south east);
}
}

\newcommand{\DEBUG}{}

\ifdefined\DEBUG%

  \def\rem#1{{\marginpar{\raggedright\scriptsize #1}}}
  \newcommand{\pmr}[1]{\rem{\color{blue}{$\bullet$ #1}}}
  \newcommand{\ppr}[1]{\rem{\color{red}{$\bullet$ #1}}}
 \else%

  \newcommand{\ppr}[1]{}
  \newcommand{\pmr}[1]{}
 \fi

\newcommand{\one}{\cdot\mathds{1}}

\newcommand{\calf}{{\mathcal F}}
\newcommand{\calg}{{\mathcal G}}

\def\rho{\varrho_1}

\def\rd{\,{\mathrm d}}

\theoremstyle{plain}
\newtheorem{theorem}{Theorem}
\newtheorem{lemma}{Lemma}
\newtheorem{fact}{Fact}
\newtheorem{corollary}{Corollary}
\newtheorem{proposition}{Proposition}

\theoremstyle{definition}
\newtheorem{remark}{Remark}

\begin{document}

\title
[Efficient approximation of SDEs]
{Efficient approximation of SDEs driven by countably dimensional Wiener process and Poisson random measure}

\author[P. Przyby{\l}owicz]{Pawe{\l} Przyby{\l}owicz}
\address{AGH University of Science and Technology,
Faculty of Applied Mathematics,
Al. A.~Mickiewicza 30, 30-059 Krak\'ow, Poland}
\email{pprzybyl@agh.edu.pl, corresponding author}

\author[M. Sobieraj]{Micha{\l} Sobieraj}
\address{AGH University of Science and Technology,
	Faculty of Applied Mathematics,
	Al. A.~Mickiewicza 30, 30-059 Krak\'ow, Poland}
\email{sobieraj@agh.edu.pl}
	
\author[{\L}. Stepień]{{\L}ukasz St\c epie\'{n}}
\address{AGH University of Science and Technology,
	Faculty of Applied Mathematics,
	Al. A.~Mickiewicza 30, 30-059 Krak\'ow, Poland}
\email{lstepie@agh.edu.pl}

\begin{abstract}
In this paper we deal with  pointwise approximation of solutions of  stochastic differential equations (SDEs) driven by infinite dimensional Wiener process with additional jumps generated by Poisson random measure. The further investigations contain upper error bounds for the proposed truncated dimension randomized Euler scheme. We also establish matching (up to constants) upper and lower bounds for $\varepsilon$-complexity and show that the defined algorithm is optimal in the Information-Based Complexity (IBC) sense. Finally, results of numerical experiments performed by using  GPU architecture are also reported. 
\newline
\newline
\textbf{Key words:} countably dimensional Wiener process, Poisson random measure, stochastic differential equations with jumps, randomized Euler algorithm, lower error bounds, complexity
\newline
\newline
\textbf{MSC 2010:} 65C30, 68Q25
\end{abstract}
\maketitle

\section{Introduction}
We investigate strong pointwise approximation of solutions of the following stochastic differential equations 
\begin{equation}
\label{main_equation}
	\left\{ \begin{array}{ll}
	\displaystyle{
	\rd X(t) = a(t,X(t))\rd t + b(t,X(t)) \rd W(t) + \int\limits_{\mathcal{E}}c(t,X(t-),y)N(\rd y,\rd t), \ t\in [0,T]},\\
	X(0)=\eta, 
	\end{array} \right.
\end{equation}
where $T >0$, $\mathcal{E} =\mathbb{R}^{d'}_0:= \mathbb{R}^{d'}\setminus \{0\}$, $d' \in \mathbb{N}$, and $W = [W_1, W_2, \ldots]^T$ is a countably dimensional Wiener process on a complete probability space $(\Omega, \Sigma, \mathbb{P})$, i.e., an infinite sequence of independent scalar Wiener processes defined on the same probability space. We also assume that $N(\rd y,\rd t)$ is a Poisson random measure with an intensity measure $\nu(\rd y)\rd t$, where $\nu(\rd y)$ is a finite L\'{e}vy measure on $(\mathcal{E},\mathcal{B}(\mathcal{E}))$. We assume that $N$ and $W$ are independent. In the sequel we will also impose suitable regularity conditions on the coefficients $a,b$ and $c$. 

Analytical properties and applications of such SDEs are widely investigated in \cite{CohEl} and \cite{Krylov1}. It follows that the case of countably dimensional Wiener process naturally extends the well-known case when only finite dimensional $W$ is considered. It allows us to model much more complicated structure (in fact, infinite dimensional) of  noise but we do not have to use (somehow involved) theory of stochastic partial differential equations. However, in most cases  exact solutions of the underlying SDEs are not known (even when only finite dimensional Wiener process $W$ is considered) and efficient approximation of solutions, together with  implementation of developed algorithms, is of main interest.

The topic of approximation of  jump-diffusion SDEs driven by finite dimensional Wiener and Poisson processes has been widely investigated in the literature in recent years, see, for example, \cite{BR}, \cite{sabanis}, \cite{deng}, \cite{deng2}, \cite{hikl1}, \cite{hikl2}, \cite{hikl3}, \cite{PPMS},  and \cite{PBL}, which is a standard book reference.  Lower error bounds and optimality issues have been raised in \cite{JDPP}, \cite{AKAPhD}, \cite{AKPP}, \cite{PP1}, \cite{PP2}. The growing popularity of SDEs with jumps follows from their wide applications  in, for example, mathematical finance, modelling energy markets etc., see \cite{PBL}, \cite{PPMSFX}, \cite{situ}.  

In this paper we define so-called  truncated dimension randomized Euler scheme $\bar X^{RE}_{M,n}$ and we use it to approximate the value of $X(T)$, where the error is measured the $L^p(\Omega)$-norm. In particular, the algorithm uses only $n$ finite dimensional evaluations of $W^M=[W_1,\ldots,W_M]^T$. This algorithm can be seen as a generalisation of the randomized Euler method investigated  in  \cite{PMPP14}, \cite{PMPP17} for SDEs driven by a finite dimensional Wiener process.  (See also  \cite{KRWU}, \cite{PMPP19} where the authors defined randomized version of the Milstein algorithm.) Recall that randomization in the drift coefficient $a=a(t,y)$ allows us to handle discontinuities wrt the time variable $t$, since we assume that $a$ is only Borel measurable in the variable $t$. 
We investigate properties of the method $\bar X^{RE}_{M,n}$ such as: error bounds and their dependence on the parameters $M,n$; its cost optimality in  certain classes of coefficients and among algorithms that use only finite number of evaluations of  $(a,b,c)$,  $W^M$, and finite number of samples from the Poisson random measure $N$. Moreover, we propose effective implementation of this algorithm in C programming language with CUDA application programming interface (API).

In summary, the main contributions of the paper are as follows:
\begin{itemize}
    \item derivation of upper error bounds for proposed truncated dimension randomized Euler algorithm $\bar X^{RE}_{M,n}$ together with its convergence rate (Theorem \ref{upper_bound}),
    \item establishment of lower error bounds and complexity bounds for numerical approximation in a particular class of algorithms (Theorems \ref{lower_b_g12}, \ref{thm_complex}),
    \item effective implementation of the method $\bar X^{RE}_{M,n}$ which utilises GPU architecture, and numerical experiments that confirm our theoretical findings (Section 5).
\end{itemize}

The structure of the paper is as follows. In Section \ref{sec:preliminaries} we describe the main problem and provide necessary notations, and definitions. In Section \ref{sec:algorithm} we define the truncated dimension Euler algorithm and show its upper error bounds. Then, in Section \ref{sec:lower_bounds} we deal with lower error bounds in the class of algorithm under consideration. We also provide complexity bounds and establish  optimality of the previously defined algorithm in some particular subclasses of the input data. Our theoretical results are supported by numerical experiments described in Section \ref{sec:experiments}. There, we also provide the key elements of our current algorithm implementation in CUDA C. Finally, auxilliary lemmas together with their proofs are presented in Appendix.
\section{Preliminaries}\label{sec:preliminaries}
Let $d, d' \geq 1, M \in \mathbb{N}\cup\{\infty\}.$ We treat a real $d$--dimensional vector as a $d\times 1$--dimensional matrix. By $\|\cdot \|$ we denote Frobenius norm for $\mathbb{R}^{d \times M}$ matrices respectively. In case $M=\infty,$ by $\|\cdot \|: \mathbb{R}^{d \times \infty} \mapsto \mathbb{R}$ we understand the Hilbert-Schmidt norm for infinite dimensional matrices. The norms appearing in the paper will be clear from the context.
We also set 
\begin{displaymath}
   \ell^2 (\mathbb{R}^d) = \{x = (x^{(1)}, x^{(2)}, \ldots)\ | \ x^{(j)}\in \mathbb{R}^d \ \hbox{for all} \ j\in\mathbb{N}, \|x\| < +\infty\},
\end{displaymath}
where $x^{(j)} = \begin{bmatrix} 
x_{1}^{(j)}\\
\vdots \\
x_{d}^{(j)} 
\end{bmatrix}$, $\displaystyle{\|x\|=\Bigl(\sum\limits_{j=1}^{+\infty}\|x^{(j)}\|^2\Bigr)^{1/2}=\Bigl(\sum\limits_{j=1}^{+\infty}\sum\limits_{k=1}^d|x_{k}^{(j)}|^2\Bigr)^{1/2} }$. 
Let $(\Omega,\Sigma,\mathbb{P})$ be a complete probability space and $\mathcal{N}_0=\{A\in\Sigma \ | \ \mathbb{P}(A)=0\}$. For a random vector $X:\Omega\mapsto\mathbb{R}^d$ we write $\|X\|_{L^p(\Omega)}=(\mathbb{E}\|X\|^p)^{1/p}$, $p\in [2,+\infty)$. Let $\nu$ be a L\'{e}vy measure on $(\mathcal{E},\mathcal{B}(\mathcal{E}))$, i.e., $\nu$ is a measure on $(\mathcal{E},\mathcal{B}(\mathcal{E}))$ that $\displaystyle{\int\limits_{\mathcal{E}}\min\{\|z\|^2,  1\}\nu(\rd z)<+\infty}$.
We further assume that $\lambda = \nu(\mathcal{E}) < +\infty$. Let $W=[W_1,W_2,\ldots]^T$ be a countably dimensional Wiener process and let $N(\rd z,\rd t)$ be a Poisson random measure, both defined on the space $(\Omega,\Sigma,\mathbb{P})$. We assume that $W$ and $N$ are independent of each other. Let $(\Sigma_t)_{t\geq  0}$ be a filtration on $(\Omega,\Sigma,\mathbb{P})$ that satisfies the usual conditions, i.e., $\mathcal{N}_0\subset\Sigma_0$ and $(\Sigma_t)_{t\geq  0}$ is right-continuous, see \cite{Protter}. We assume that $W$ is an $(\Sigma_t)_{t\geq  0}$-Wiener process and $N(\rd z,\rd t)$ is an $(\Sigma_t)_{t\geq  0}$-Poisson measure with the intensity measure $\nu(\rd z)\rd t$.
  Then, by Theorem 1.4.1 in \cite{Kunita}, there exists a scalar  Poisson process $N=(N(t))_{t\geq 0}$ with intensity $\lambda$ and an iid sequence of $\mathcal{E}$-valued random variables $(\xi_k)_{k=1}^{+\infty}$ with the common distribution $\nu(\rd y)/\lambda$ such that the Poisson random measure $N(\rd z,\rd t)$ can be written as follows
$\displaystyle{N(E\times (s,t])=\sum\limits_{N(s) < k\leq N(t)}\mathbf{1}_E(\xi_k)}$ for $0\leq s<t\leq T,  E\in\mathcal{B}(\mathcal{E})$. The $d'$-dimensional compound Poisson process, associated with the Poisson measure $N$, is defined as $\displaystyle{L(t)=\sum\limits_{k=1}^{N(t)}\xi_k=\int\limits_0^t\int\limits_{\mathcal{E}}yN(\rd y,\rd s)}$. We also set $\displaystyle{\Sigma_{\infty}=\sigma\Bigl(\bigcup_{t\geq 0}\Sigma_t\Bigr)}$, $\displaystyle{\Sigma_{\infty}^Z=\sigma\Bigl(\bigcup_{t\geq 0}\sigma(Z(t))\Bigr)}$, $Z\in\{N,L,W_1,W_2,\ldots\}$, and  we define
$\displaystyle{    \mathcal{H}_M=\sigma\Bigl(\mathcal{N}_0\cup\bigcup_{k=1,\ldots,M} \Sigma_{\infty}^{W_k}\Bigr)}$, $\displaystyle{ \mathcal{H}^+_M=\sigma\Bigl(\mathcal{N}_0\cup\bigcup_{k\geq M+1} \Sigma_{\infty}^{W_k}\Bigr)}$ for any $M\in\mathbb{N}$.
Of course $\mathcal{H}_M$ and $\mathcal{H}^+_M$ are independent for any $M\in\mathbb{N}$. Moreover, for all $M\in\mathbb{N}$, $Z\in\{N,L\}$ the $\sigma$-fields $\sigma(\mathcal{H}_M\cup\mathcal{H}^+_M)$ and $\Sigma^Z_{\infty}$ are independent. For a c\`adl\`ag process $(Y(t))_{t\in [0,T]}$ by $(Y(t-))_{t\in [0,T]}$ we mean its c\`agl\`ad modification. We refer to Chapter 2.9. in \cite{applebaum} for further properties of c\`adl\`ag mappings. 

For $D,L>0$ we consider $\mathcal{A}(D,L)$ a class of all functions $a: [0,T]\times \mathbb{R}^d \mapsto \mathbb{R}^d $ satisfying the following conditions:
\begin{enumerate}
	\item [(A1)] $a$ is Borel measurable,
	\item [(A2)] $\|a(t,0)\|\leq D$ for all $t\in[0,T]$,
	\item [(A3)] $\|a(t,x) - a(t,y)\| \leq L\|x-y\|$ for all $x,y \in \mathbb{R}^d$, $t \in [0,T]$.
\end{enumerate}
Let $\Delta = (\delta(k))_{k = 1}^{+\infty} \subset \mathbb{R}$ be a positive, strictly decreasing sequence, converging to zero and let $C>0$,  $\varrho_1 \in (0,1]$.  We consider the following class $\mathcal{B}(C,D,L,\Delta,\varrho_1)$ of functions $b = (b^{(1)}, b^{(2)}, \ldots):[0,T] \times \mathbb{R}^d\mapsto\ell^2 (\mathbb{R}^d)$, where $\ b^{(j)}:[0,T] \times \mathbb{R}^d  \mapsto \mathbb{R}^d$, $j\in\mathbb{N}$. Namely, $b\in \mathcal{B}(C,D,L,\Delta,\varrho_1)$ iff it satisfies the following conditions:
\begin{enumerate}
	\item [(B1)] $\Vert b(0,0) \Vert \leq D$,
	\item [(B2)] $\Vert b(t,x) - b(s,x)\Vert \leq L(1+ \|x\|)|t-s|^{\varrho_1}$ for all $x \in \mathbb{R}^d$ and $t,s\in [0,T]$,
	\item [(B3)] $\Vert b(t,x) - b(t,y)\Vert \leq L \|x-y\|$	for all $x,y \in \mathbb{R}^d$ and $t \in [0,T]$,
	\item [(B4)] $\sup_{0\leq t \leq T}\Vert b(t,x) - P_k b(t,x)\Vert \leq C(1+ \|x\|)\delta(k)$ for all $k \in \mathbb{N}$ and $x\in \mathbb{R}$, 
	where $P_k:  \ell^2 (\mathbb{R}^d) \mapsto  \ell^2 (\mathbb{R}^d)$ is the following projection operator 
	\begin{equation*}
	P_k x = (x^{(1)}, x^{(2)}, \ldots, x^{(k)}, 0, 0,  \ldots), \quad   x\in\ell^2(\mathbb{R}^d).
	\end{equation*}
	We denote $b^k=P_kb$ and then $b^k(t,y)=P_k(b(t,y))$ for all $(t,y)\in [0,T]\times\mathbb{R}^d$. We also set  $P_{\infty}=Id$, so $P_{\infty}x=x$ for all $x\in\ell^2(\mathbb{R}^d)$.
\end{enumerate}
Let $p\in [2,+\infty)$, $\varrho_{2}\in (0,1]$ and let $\nu$ be the L\'{e}vy measure as above. We say that a function $c: [0,T]\times \mathbb{R}^d \times \mathbb{R}^{d'} \mapsto \mathbb{R}^d$ belongs to the class $\mathcal{C}(p,D,L,\varrho_2, \nu)$ if and only if
\begin{enumerate}
    \item [(C1)] $c$ is Borel measurable,
	\item [(C2)] $\displaystyle{\biggr(\int\limits_{\mathcal{E}}\|c(0,0,y)\|^p \nu(\rd y)\biggr)^{1/p} \leq D}$,
	\item [(C3)] $\displaystyle{\biggr(\int\limits_{\mathcal{E}}\|c(t,x_1,y) - c(t,x_2,y)\|^p \  \nu(\rd y)\biggr)^{1/p} \leq L\|x_1-x_2\|}$ for all $x_1,x_2 \in \mathbb{R}^d$, $t \in [0,T]$,
	\item [(C4)] $\displaystyle{\biggr(\int\limits_{\mathcal{E}}\|c(t_1,x,y) - c(t_2,x,y)\|^p  \ \nu(\rd y)\biggr)^{1/p} \leq L(1+\|x\|)|t_1-t_2|^{\varrho_2}}$ for all $x\in \mathbb{R}^d$, $t_1,t_2 \in [0,T]$.
\end{enumerate}
Finally,  we define the following class
\begin{equation*}
	\mathcal{J}(p,D) = \{\eta \in L^p(\Omega) \ |  \ \sigma(\eta)\subset\Sigma_0, \Vert \eta \Vert_{L^p(\Omega)}\leq D\}.
\end{equation*}
As a set of admissible input data we consider the following class
\begin{equation*}
	\calf(p, C,D,L,\Delta,\varrho_1, \varrho_2, \nu) = \mathcal{A}(D,L) \times \mathcal{B}(C,D,L,\Delta,\varrho_1)\times \mathcal{C}(p,D,L,\varrho_2, \nu) \times \mathcal{J}(p,D).
\end{equation*}
The constants $T,d,d',\lambda$ together with $p,C,D,L,\varrho_1, \varrho_2$, the L\'evy measure $\nu$, and the sequence $\Delta$ are referred to as parameters of the  class $\calf(p, C,D,L,\Delta,\varrho_1, \varrho_2, \nu)$. Except for $T,d,d',\lambda,\nu$, the parameters are not known and cannot be used by an algorithm as input parameters.

Since  $N(\rd y,\rd s)$ is a finite random (counting) measure, and $X(s)$, $X(s-)$ differ on at most countable number of time points, the equation \eqref{main_equation} can be written as
\begin{equation}
\label{main_2}
	X(t) = \eta + \int\limits_{0}^{t}\tilde a(s,X(s))\rd s + \int\limits_{0}^{t}b(s,X(s)) \rd W(s)+\int\limits_{0}^{t}\int\limits_{\mathcal{E}}c(s,X(s-),y)\tilde N(\rd y,\rd s),
\end{equation}
where 
\begin{equation*}
    \tilde N(\rd y,\rd t) = N(\rd y,\rd t)-\nu(\rd y)\rd t
\end{equation*}
is the compensated Poisson measure, and 
\begin{eqnarray}
\label{ta_def}
        \tilde a(t,x)=a(t,x)+\int\limits_{\mathcal{E}}c(t,x,y)\nu(\rd y).
\end{eqnarray}
Moreover,
\begin{equation*}
    \int\limits_{0}^{t}b(s,X(s))\rd W(s)=\sum\limits_{j=1}^{+\infty}\int\limits_0^t b^{(j)}(s,X(s))\rd W_{j}(s)
\end{equation*}
is the stochastic It\^o integral wrt the countable dimensional Wiener process $W$, see pages 427-428 in \cite{CohEl}. (See also \cite{Krylov1} where even more general setting than \eqref{main_2} is considered.) In the case when $W$ is countably dimensional Wiener process the above stochastic It\^o integral  can be understood as a stochastic integral wrt  cylindrical Wiener process in the Hilbert space $\ell^2$, see pages 289-290  in \cite{CohEl} and Remark 3.9 in \cite{dalang}. Alternatively properties of such stochastic integrals were widely described and proved in \cite{Cao}, \cite{Liang06}. From \cite{CohEl} and the  papers \cite{Cao}, \cite{Liang06} it follows that  the stochastic It\^o integral wrt the countably dimensional Wiener process $W$ has analogous properties as in the case finite dimensional case -  in particular, we can use the Burkholder's inequality, It\^o formula etc. Moreover, by the Fact \ref{fact_1} and Lemma 17.1.1 in \cite{CohEl}  for any $(a,b,c,\eta)\in\calf(p, C,D,L,\Delta,\varrho_1, \varrho_2, \nu)$ there exists a unique strong solution $X=X(a,b,c,\eta)$ of the equation \eqref{main_2} (and therefore also \eqref{main_equation}).

The aim of this paper is to construct an efficient  scheme that approximates $X(T)$, i.e., the value of solution of \eqref{main_equation} at the final time point $T$. We consider algorithms that use only finite dimensional evaluations of $W$ at finite number of points in $[0,T]$. The idea of approximating the solution of $X$ is as follows. For a fixed $M\in \mathbb{N}$ as a first approximation of $X$ we use the process $X^M$ -  a unique strong  solution of the following SDE
\begin{eqnarray}
\label{main_equation_findim}
	&&X^M(t) = \eta + \int\limits_{0}^{t}a(s,X^M(s))\rd s + \int\limits_{0}^{t}b^M (s,X^M(s))\rd W(s)\notag\\
	&&\quad\quad\quad\quad\quad\quad+\int\limits_{0}^{t}\int\limits_{\mathcal{E}}c(s,X^M(s-),y)N(\rd y,\rd s), \quad t\in [0,T].
\end{eqnarray}
Since for any $M\in\mathbb{N}$
\begin{equation*}
    \int\limits_0^t b^M(s,X^M(s))\rd W(s)=\Bigl[\sum\limits_{j=1}^M\int\limits_0^t b_{k}^{(j)}(s,X^M(s))\rd W_j(s)\Bigr]_{k=1,2,\ldots,d},
\end{equation*}
the SDE \eqref{main_equation_findim} can be equivalently viewed  as a finite dimensional SDE   driven by the  $M$-dimensional Wiener process $W^M=[W_1,W_2,\ldots,W_M]^T$. Again, by rewriting the equation \eqref{main_equation_findim} analogously as in \eqref{main_2} we get by Fact  \ref{fact_1} and Lemma 17.1.1 in \cite{CohEl}  that for all $(a,b,c,\eta)\in\calf(p, C,D,L,\Delta,\varrho_1, \varrho_2, \nu) $ and every $M\in\mathbb{N}\cup\{\infty\}$ the equation \eqref{main_equation_findim} has a unique strong solution $X^M=X^M(a,b,c,\eta)$. From the uniqueness of solution  for any $M\in\mathbb{N}$ we have $X^M(a,b,c,\eta)=X(a,b^M,c,\eta)$, and  $X^{\infty}=X$ for $M=\infty$. 

In the following lemma  and proposition we gathered results on moments bound,  $L^p(\Omega)$-regularity of $X^M$, and main approximation property of $X^M$. The proofs are postponed to the Appendix. 
\begin{lemma}
\label{lemma_sol}
	There exist $C_1, C_2 \in (0,+\infty),$ depending only on the parameters of the class $\calf(p,C,D,L,\Delta,\varrho_1, \varrho_2, \nu),$ such that for every 
	$M\in\mathbb{N}\  \cup \ \{\infty\}$ and 
	$(a,b,c,\eta) \in \calf(p, C,D,L,\Delta,\varrho_1, \varrho_2, \nu) $ we have
	\begin{equation}
	\label{lemma_sol_estimate}
		\mathbb{E}\Bigl(\sup_{0 \leq t \leq T} \|X^M(t)\|^p\Bigr) \leq C_1
	\end{equation}
	and for all $s,t \in [0,T]$ the following holds:
	\begin{itemize}
	    \item [(i)] if $c \not \equiv 0$ then
	    \begin{equation*}
		    \mathbb{E}\|X^M(t)-X^M(s)\|^p \leq C_2|t-s|,
	    \end{equation*}
	\item [(ii)]  if $b \not \equiv 0$ and $c = 0$ then
	\begin{equation*}
	    \mathbb{E}\|X^M(t)-X^M(s)\|^p \leq C_2|t-s|^{p/2},
	\end{equation*}
	\item [(iii)] if $b = 0$ and $c = 0$ then
	\begin{equation*}
	    \mathbb{E}\|X^M(t)-X^M(s)\|^p \leq C_2|t-s|^{p}.
	\end{equation*}
	\end{itemize}
\end{lemma}
Note that the cases (i) and (ii) coincide only when $p=2$. Moreover, we stress that $C_1,C_2$ in Lemma \ref{lemma_sol} do not depend on the truncation parameter $M$.
\begin{proposition}
\label{aux_lem_1}
  There exist $K_1, K_2\in (0,+\infty)$, $M_0\in\mathbb{N}$ such that for any $M \in \mathbb{N}$  it holds
 \begin{equation}
 \label{xxn_est_1}
     	\sup\limits_{(a,b,c,\eta)\in\calf(p,C,D,L,\Delta,\varrho_1, \varrho_2, \nu)}\sup\limits_{0 \leq t \leq T}\Vert X(a,b,c,\eta)(t)-X^M(a,b,c,\eta)(t)\Vert_{L^p(\Omega)} \leq K_1\delta(M),
 \end{equation}
 and for any $M\geq M_0$ we have
  \begin{equation}
  \label{xxn_est_2}
     	\sup\limits_{(a,b,c,\eta)\in\calf(p,C,D,L,\Delta,\varrho_1, \varrho_2, \nu)}\sup\limits_{0 \leq t \leq T}\Vert X(a,b,c,\eta)(t)-X^M(a,b,c,\eta)(t)\Vert_{L^2(\Omega)}\geq \frac{1}{2}K_2T^{1/2}\delta(M).
 \end{equation}
Hence,
\begin{equation}
\label{xxn_est_3}
    \sup\limits_{(a,b,c,\eta)\in\calf(p,C,D,L,\Delta,\varrho_1, \varrho_2, \nu)}\sup\limits_{0 \leq t \leq T}\Vert X(a,b,c,\eta)(t)-X^M(a,b,c,\eta)(t)\Vert_{L^p(\Omega)}=\Theta(\delta(M)),
\end{equation}
as $M\to+\infty$.
\end{proposition}
The next step is to approximate  $X^M$ via  truncated dimension randomized Euler scheme $\bar X^{RE}_{M,n}$. In the next section we provide definition and analysis of the algorithm $\bar X^{RE}_{M,n}$. Moreover, in some subclasses of $\calf(p,C,D,L,\Delta,\varrho_1, \varrho_2, \nu)$ we establish complexity bounds and optimality of the truncated dimension randomized Euler scheme by using Information-Based Complexity (IBC) framework, see \cite{TWW88}. Finally, we show efficient implementation of the method $\bar X^{RE}_{M,n}$ and present result of numerical experiments performed on GPU.

Unless otherwise stated all constants appearing in   estimates and in the "O", "$\Omega$", "$\Theta$" notation will only depend on the parameters of the class $\calf(p,C,D,L,\Delta,\varrho_1, \varrho_2, \nu)$. Moreover, the same letter might be used to denote different constants.
\section{Truncated dimension randomized Euler algorithm}\label{sec:algorithm}

We define the truncated dimension randomized Euler algorithm that approximates the value of $X(T)$. Let $M,n \in \mathbb{N}$,  $t_j = jT/n$, $j=0,1,\ldots, n$. We denote by
 $\Delta W_{j} = [\Delta W_{j,1}, \Delta W_{j,2}, \ldots]^{T},$ where $ \Delta W_{j,k} = W_k(t_{j+1}) - W_k(t_j)$ for $k\in\mathbb{N}$. Let $\left(\theta_j\right)_{j=1}^{n-1}$ be a sequence of independent random variables, where each $\theta_j$ is uniformly distributed on $[t_j,t_{j+1}]$, $j=0,1,\ldots,n-1$. We also assume that $\displaystyle{\sigma(\theta_0,\theta_1,\ldots,\theta_{n-1})}$ is  independent of $\Sigma_\infty$. For $(a,b,c,\eta)\in\calf(p, C,D,L,\Delta,\varrho_1, \varrho_2, \nu)$ we set
\begin{equation}\label{main_scheme}
	\begin{cases}
		X_{M,n}^{RE}(0) = \eta \\
		X_{M,n}^{RE}(t_{j+1}) = X_{M,n}^{RE}(t_{j}) + a(\theta_j , X_{M,n}^{RE}(t_{j}))\frac{T}{n} + b^M(t_j, X_{M,n}^{RE}(t_{j}))\Delta W_j\\  \quad\quad\quad\quad\quad\quad +\sum\limits_{k=N(t_j)+1}^{N(t_{j+1})}c(t_j,X_{M,n}^{RE}(t_j), \xi_k), \quad j=0,1,\ldots, n-1.
	\end{cases}
\end{equation}
The truncated dimension randomized Euler algorithm $\bar X^{RE}_{M,n}$ is defined by
\begin{equation*}
    \bar X^{RE}_{M,n}(a,b,c,\eta)=X^{RE}_{M,n}(T).
\end{equation*}

In order to analyse the truncated dimension randomized Euler scheme $X^{RE}_{M,n}$ 
we define its time-continuous version denoted by $\tilde{X}_{M,n}^{RE}$. Set
\begin{equation*}
	\tilde{X}_{M,n}^{RE}(0) = \eta
\end{equation*}
and 
\begin{equation}
\label{main_scheme_continuous}
	\begin{split}
	\tilde{X}_{M,n}^{RE}(t) = \tilde{X}_{M,n}^{RE}(t_j) +
	a(\theta_j , \tilde{X}_{M,n}^{RE}(t_{j}))(t-t_j) & + b^M(t_j, \tilde{X}_{M,n}^{RE}(t_{j}))(W(t) - W(t_j)) \\
	& + \int\limits_{t_j}^{t}\int\limits_{\mathcal{E}}c(t_j, \tilde{X}_{M,n}^{RE}(t_j), y)N(\rd y,\rd s)
	\end{split}
\end{equation}
for $t \in [t_j, t_{j+1}], \hspace{0.2cm} j=0,1,\ldots, n-1$. Due to the fact that $\nu$ is a finite L\'{e}vy measure, there are only finitely many jumps of $(N(t))_{t\in [0,T]}$ in every subinterval $[t_j,t_{j+1}]$ and
\begin{equation*}
	 \sum_{k=N(t_j)+1}^{N(t)}c(t_j,\tilde X_{M,n}^{RE}(t_j),\xi_k)=\int\limits_{t_j}^{t}\int\limits_{\mathcal{E}}c(t_j,\tilde X_{M,n}^{RE}(t_j),y)N(\rd y,\rd s),
\end{equation*}
for $t\in [t_j,t_{j+1}]$. Hence, it can be  shown by induction that
\begin{equation}\label{schemes_coincide}
	\tilde{X}_{M,n}^{RE}(t_j) = X_{M,n}^{RE}(t_j),  \quad j=0,1,\ldots, n.
\end{equation}
Note that the trajectories of  $\tilde X^{RE}_{M,n}=(\tilde X^{RE}_{M,n}(t))_{t\in [0,T]}$ are c\`adl\`ag. 
As in \cite{PMPP17} we consider the extended filtration $(\tilde \Sigma^n_t)_{t\geq 0}$, where $
    \tilde\Sigma^n_t=\sigma\Bigl(\Sigma_t\cup\sigma(\theta_0,\theta_1,\ldots,\theta_{n-1})\Bigr)$. Since $\Sigma_{\infty}$ and $\sigma(\theta_0,\theta_1,\ldots,\theta_{n-1})$ are independent, the process $W$ is still $(\tilde\Sigma^n_t)_{t\geq 0}$-Wiener process while $N(\rd z,\rd t)$ is $(\tilde\Sigma^n_t)_{t\geq 0}$-Poisson random measure. 
\begin{lemma}
\label{prop_tcRE}
    Let $M,n\in\mathbb{N}$ and $(a,b,c,\eta)\in\calf(p, C,D,L,\Delta,\varrho_1, \varrho_2, \nu)$. Then the process $\tilde X^{RE}_{M,n}=(\tilde X^{RE}_{M,n}(t))_{t\in [0,T]}$ is $(\tilde \Sigma^n_t)_{t\in [0,T]}$-progressively measurable.
\end{lemma}
The proof easily follows from induction and the well-known fact that adapted c\`adl\`ag processes are progressive.
We now state the upper error bound on the error of the truncated dimension randomized Euler algorithm.
\begin{theorem}
\label{upper_bound}
	 There exists a positive constant $K$, depending only on the parameters of the class $\calf(p,C,D,L,\Delta,\varrho_1, \varrho_2, \nu)$, such that for every $M,n\in\mathbb{N}$ and $(a,b,c, \eta)\in\calf(p, C,D,L,\Delta,\varrho_1, \varrho_2, \nu)$ it holds
	\begin{equation*}
	    \|X(a,b,c,\eta)(T)-\bar X^{RE}_{M,n}(a,b,c,\eta)\|_{L^p(\Omega)}\leq K\Bigl( n^{-\min\{\varrho_1, \varrho_2, 1/p\}} +  \delta(M)\Bigr).
	\end{equation*}
\end{theorem}
For the proof of Theorem \ref{upper_bound} we need the following result.
\begin{proposition}
\label{aux_lem_2}
There exists a positive constant $C_0$, depending only on the parameters of the input data class $\calf(p, C,D,L,\Delta,\varrho_1, \varrho_2, \nu)$, such that for every $M,n\in\mathbb{N}$ and $(a,b,c,\eta)\in\calf(p, C,D,L,\Delta,\varrho_1, \varrho_2, \nu)$ it holds
	\begin{equation*}
		\sup_{0 \leq t \leq T} \Vert\tilde{X}_{M,n}^{RE}(t)-X^M(t)\Vert_{L^p(\Omega)} \leq C_0 n^{-\min\{\varrho_1, \varrho_2, 1/p\}}.
	\end{equation*}
	In particular, if $b=0$ and $c\not\equiv 0$ then
	\begin{equation*}
		\sup_{0 \leq t \leq T} \Vert\tilde{X}_{M,n}^{RE}(t)-X^M(t)\Vert_{L^p(\Omega)} \leq C_0 n^{-\min\{\varrho_2, 1/p\}},
	\end{equation*} 
	when  $b \not\equiv 0$, $c=0$ we have
	\begin{equation*}
		\sup_{0 \leq t \leq T} \Vert\tilde{X}_{M,n}^{RE}(t)-X^M(t)\Vert_{L^p(\Omega)} \leq C_0 n^{-\min\{\varrho_1, 1/2\}},
	\end{equation*}
	while if $b=0$, $c=0$ then 
	\begin{equation*}
		\sup_{0 \leq t \leq T} \Vert\tilde{X}_{M,n}^{RE}(t)-X^M(t)\Vert_{L^p(\Omega)} \leq C_0 n^{-1/2}.
	\end{equation*}
\end{proposition}
\begin{proof}
We deliver the proof in general case with drift, diffusion and jump coefficients being non-zero.  Firstly, we can rewrite \eqref{main_equation_findim} for all $t\in [0,T]$ as follows
	\begin{equation}\label{equivalent_continuous_scheme}
		X^M(t) = \eta + \int\limits_{0}^{t}\hat{a}(s)\rd s + \int\limits_{0}^{t}\hat{b}^M(s)\rd W(s) + \int\limits_{0}^{t}\int\limits_{\mathcal{E}}\hat{c}(y,s)N(\rd y,\rd s), 
	\end{equation}
	with
	\begin{equation*}
		\hat{f}(s) = \sum_{j=0}^{n-1}f(s,X^M(s))\one_{(t_j,t_{j+1}]}(s), \quad f \in \{a, b^M\},
	\end{equation*}
	\begin{equation*}
		\hat{c}(y,s) = \sum_{j=0}^{n-1}c(s,X^M(s-),y)\one_{\mathcal{E} \times (t_j,t_{j+1}]}(y, s).
	\end{equation*}
	Moreover, we define three auxiliary functions
\begin{equation*}
	\tilde{a}_{M,n}(s) = \sum_{j=0}^{n-1}a(\theta_j,\tilde{X}_{M,n}^{RE}(t_j))\one_{(t_j,t_{j+1}]}(s),
\end{equation*}
\begin{equation*}
	\tilde{b}_n^M(s) = \sum_{j=0}^{n-1}b^M(t_j,\tilde{X}_{M,n}^{RE}(t_j))\one_{(t_j,t_{j+1}]}(s),
\end{equation*}
\begin{equation*}
	\tilde{c}_{M,n}(y,s) = \sum_{j=0}^{n-1} c(t_j, \tilde{X}_{M,n}^{RE}(t_j),y)\one_{\mathcal{E} \times (t_j, t_{j+1}]}(y,s),
\end{equation*}
and by \eqref{main_scheme_continuous} we have for all $t\in [0,T]$ that
\begin{equation}\label{lemma_X_RE_equation}
	\tilde{X}_{M,n}^{RE}(t) = \eta + \int\limits_{0}^{t}\tilde{a}_{M,n}(s) \rd s + \int\limits_{0}^{t}\tilde{b}_n^M(s)\rd W(s) + \int\limits_{0}^t \int\limits_\mathcal{E}\tilde{c}_{M,n}(y,s)N(\rd y,\rd s).
\end{equation}
Due to Lemma \ref{prop_tcRE} all stochastic integrals involved in \eqref{lemma_X_RE_equation} are well-defined. By \eqref{equivalent_continuous_scheme} and \eqref{lemma_X_RE_equation}, we get for $t\in[0,T]$ that
	\begin{equation}\label{lemma_AB_RE}
		\mathbb{E}\|X^M(t) - \tilde{X}_{M,n}^{RE}(t)\|^p \leq 3^{p-1}(\mathbb{E}\|A_n^M(t)\|^p + \mathbb{E}\|B_n^M(t)\|^p + \mathbb{E}\|C_n^M(t)\|^p),
	\end{equation}
	where 
	\begin{equation*}
		\mathbb{E}\|A_n^M(t)\|^p = \mathbb{E}\biggr\|\int\limits_{0}^{t}\big(\hat{a}(s) - \tilde{a}_{M,n} (s)\big)\rd s\biggr\|^p,
	\end{equation*}
	\begin{equation*}
		\mathbb{E}\|B_n^M(t)\|^p = \mathbb{E}\biggr\|\int\limits_{0}^{t}\big(\hat{b}^M(s) - \tilde{b}_n^M (s)\big)\rd W(s)\biggr\|^p,
	\end{equation*}
	\begin{equation*}
		\mathbb{E}\|C_n^M(t)\|^p = \mathbb{E}\biggr\|\int\limits_{0}^{t}\int\limits_{\mathcal{E}}\big(\hat{c}(y,s) - \tilde{c}_{M,n} (y,s)\big)N(\rd y,\rd s)\biggr\|^p.
	\end{equation*}
	Moreover, we have for all $t\in  [0,T]$
	\begin{equation*}
		\mathbb{E}\|A_n^M(t)\|^p \leq 3^{p-1}\big(\mathbb{E}\|A_{1,n}^{RE,M}(t)\|^p + \mathbb{E}\|A_{2,n}^{RE,M}(t)\|^p + \mathbb{E}\|A_{3,n}^{RE,M}(t)\|^p\big),
	\end{equation*}
	where
	\begin{equation*}
		\mathbb{E}\|A_{1,n}^{RE,M}(t)\|^p = \mathbb{E}\biggr\|\int\limits_{0}^{t}\sum_{j=0}^{n-1}\big(a(s,X^M(s)) - a(s,X^M(t_j))\big)\one_{(t_j,t_{j+1}]}(s)\rd s\biggr\|^p,
	\end{equation*}
	\begin{equation*}
		\mathbb{E}\|A_{2,n}^{RE,M}(t)\|^p = \mathbb{E}\biggr\|\int\limits_{0}^{t}\sum_{j=0}^{n-1}\big(a(s,X^M(t_j)) - a(\theta_j,X^M(t_j))\big)\one_{(t_j,t_{j+1}]}(s)\rd s\biggr\|^p,
	\end{equation*}
	\begin{equation*}
		\mathbb{E}\|A_{3,n}^{RE,M}(t)\|^p = \mathbb{E}\biggr\|\int\limits_{0}^{t}\sum_{j=0}^{n-1}\big(a(\theta_j,X^M(t_j)) - a(\theta_j, \tilde{X}_{M,n}^{RE}(t_j))\big)\one_{(t_j,t_{j+1}]}(s)\rd s\biggr\|^p.
	\end{equation*}
	The Lemma \ref{lemma_sol} together with the H\"{o}lder inequality yields
	\begin{equation*}
			\mathbb{E}\|A_{1,n}^{RE,M}(t)\|^p 
			 \leq T^{p-1}L^p\sum_{j=0}^{n-1}\int\limits_{t_j}^{t_{j+1}}\mathbb{E}\|X^M(s) - X^M(t_j)\|^p \rd s \leq K_1 n^{-1}.
	\end{equation*}
	Proceeding analogously as in the proof of inequality  (71) in \cite{PMPP14} we get  by Lemma \ref{lemma_sol} that
	\begin{equation*}
		\mathbb{E}\|A_{2,n}^{RE,M}(t)\|^p \leq K_2 n^{-p/2}.
	\end{equation*}
	Again by the H\"{o}lder inequality for every $t \in [0,T]$ it holds
	\begin{equation*}
			\mathbb{E}\|A_{3,n}^{RE,M}(t)\|^p 
			 \leq K_3 \int\limits_{0}^{t}\sum_{j=0}^{n-1}\mathbb{E}\big\|X^M(t_j) - \tilde{X}_{M,n}^{RE}(t_j)\big\|^p \one_{(t_j,t_{j+1}]}(s)\rd s.
	\end{equation*}
	Finally, we obtain the following estimate
	\begin{equation}
	\label{lemma_A1_RE}
		\mathbb{E}\|A_n^M(t)\|^p \leq K_4 \int\limits_{0}^{t}\sup_{0 \leq u \leq s}\mathbb{E}\big\|X^M(u) - \tilde{X}_{M,n}^{RE}(u)\big\|^p \rd s + K_5 n^{-1}.
	\end{equation}
	By the Burkholder inequality we have for every $t \in [0,T]$ that
	\begin{equation*}
		\mathbb{E}\|B_n^M(t)\|^p \leq K_6\int\limits_0^t\|\hat b^M(s)-\tilde b^M_n(s)\|^p \rd s. 
	\end{equation*}
	Therefore
	\begin{equation}
	\label{lemma_B_RE_ineq}
	    \mathbb{E}\|B_n^M(t)\|^p \leq K_7\Bigl(\mathbb{E}\big[B_{1,n}^{RE,M}(t)\big] + \mathbb{E}\big[B_{2,n}^{RE,M}(t)\big] + \mathbb{E}\big[B_{3,n}^{RE, M}(t)\big]\Bigr),
	\end{equation}
	where 
	\begin{equation*}
		\mathbb{E}\big[B_{1,n}^{RE, M}(t)\big] =  \mathbb{E}\int\limits_{0}^{t}\sum_{j=0}^{n-1}\|b^M(s,X^M(s)) - b^M(s,X^M(t_j))\|^p\one_{(t_j,t_{j+1}]}(s)\rd s, 
	\end{equation*}
	\begin{equation*}
		\mathbb{E}\big[B_{2,n}^{RE, M}(t)\big] = \mathbb{E}\int\limits_{0}^{t}\sum_{j=0}^{n-1}\|b^M(s,X^M(t_j)) - b^M(t_j,X^M(t_j))\|^p\one_{(t_j,t_{j+1}]}(s)\rd s,
	\end{equation*}
	\begin{equation*}
		\mathbb{E}\big[B_{3,n}^{RE, M}(t)\big] = \mathbb{E}\int\limits_{0}^{t}\sum_{j=0}^{n-1}\|b^M(t_j,X^M(t_j)) - b^M(t_j,\tilde{X}_{M,n}^{RE}(t_j))\|^p\one_{(t_j,t_{j+1}]}(s)\rd s.
	\end{equation*}
	By Fact \ref{fact_1} and  Lemma \ref{lemma_sol} (iii) we get
	\begin{equation}
	\label{lemma_B1_RE}
		\begin{split}
			\mathbb{E}\big[B_{1,n}^{RE, M}(t)\big] \leq K_8 \mathbb{E}\int\limits_{0}^{t}\sum_{j=0}^{n-1}\|X^M(s) - X^M(t_j)\|^p\one_{(t_j,t_{j+1}]}(s) \rd s\leq K_9 n^{-1}.
		\end{split}
	\end{equation}
  Moreover,
	\begin{equation}
	\label{lemma_B2_RE}
			\mathbb{E}\big[B_{2,n}^{RE, M}(t)\big]\leq K_{10}\mathbb{E}\int\limits_0^t\sum_{j=0}^{n-1}(1+\|X^M(t_j)\|^p)(s-t_j)^{p\varrho_1}\one_{(t_j,t_{j+1}]}(s)\rd s \leq K_{11} n^{-p \varrho_1}.
	\end{equation}
Finally for $t\in [0,T]$ we get
	\begin{equation}
	\label{lemma_B3_RE}
		\mathbb{E}\big[B_{3,n}^{RE, M}(t)\big] \leq K_{12} \int\limits_{0}^{t}\sum_{j=0}^{n-1}\mathbb{E}\|X^M(t_j) - \tilde{X}_{M,n}^{RE}(t_j)\|^p \one_{(t_j,t_{j+1}]}(s)\rd s.
	\end{equation}
 Combining \eqref{lemma_B_RE_ineq}, \eqref{lemma_B1_RE}, \eqref{lemma_B2_RE}, and \eqref{lemma_B3_RE} we have for $t\in [0,T]$ that
	\begin{equation}\label{lemma_B_RE_final}
		\mathbb{E}\|B_n^M(t)\|^p \leq K_{13}(n^{-1} + n^{-p 
			\varrho_1}) + K_{14} \int\limits_{0}^{t}\sup_{0 \leq u \leq s}\mathbb{E}\|X^M(u) - \tilde{X}_{M,n}^{RE}(u)\|^p \rd s.
	\end{equation}
Now we estimate the jump part in \eqref{lemma_AB_RE}. By the  Kunita inequality (see, for example, Theorem 2.11 in \cite{Kunita1})  we have
	\begin{eqnarray}
	\label{lemma_C_compensator}
		&&\mathbb{E}\|C_n^M(t)\|^p \leq K_{15}\mathbb{E}\biggr\|\int\limits_{0}^{t}\int\limits_{\mathcal{E}}\big(\hat{c}(y,s) - \tilde{c}_{M,n} (y,s)\big)\tilde{N}(\rd y,\rd s)\biggr\|^p\notag\\ 
		&&+ K_{15}\mathbb{E}\biggr\|\int\limits_{0}^{t}\int\limits_{\mathcal{E}}\big(\hat{c}(y,s) - \tilde{c}_{M,n} (y,s)\big)\nu(\rd y)\rd s\biggr\|^p\leq K_{16} \mathbb{E}\int\limits_{0}^{t}\int\limits_{\mathcal{E}}\|\hat{c}(y,s) - \tilde{c}_{M,n} (y,s)\|^p\nu(\rd y)\rd s.\notag
	\end{eqnarray}
 Hence, we have for all $t\in [0,T]$ that
	\begin{equation*}
		\mathbb{E}\|C_n^M(t)\|^p \leq K_2\left(\mathbb{E}\big[C_{1,n}^{RE, M}(t)\big] + \mathbb{E}\big[C_{2,n}^{RE, M}(t)\big] + \mathbb{E}\big[C_{3,n}^{RE, M}(t)\big]\right),
	\end{equation*}
	with
	\begin{displaymath}
		\mathbb{E}[C_{1,n}^{RE,M}(t)] =  \mathbb{E}\int\limits_{0}^{t}\int\limits_{\mathcal{E}}\sum_{j=0}^{n-1}\|c(s,X^M(s-),y) - c(s,X^M(t_j),y)\|^p\one_{\mathcal{E} \times(t_j,t_{j+1}]}(y,s)\nu(\rd y)\rd s, 
	\end{displaymath}
	\begin{displaymath}
		\mathbb{E}[C_{2,n}^{RE,M}(t)] = \mathbb{E}\int\limits_{0}^{t}\int\limits_{\mathcal{E}}\sum_{j=0}^{n-1}\|c(s,X^M(t_j),y) - c(t_j,X^M(t_j),y)\|^p\one_{\mathcal{E}\times(t_j,t_{j+1}]}(y,s)\nu(\rd y)\rd s,
	\end{displaymath}
	\begin{displaymath}
		\mathbb{E}[C_{3,n}^{RE,M}(t)] = \mathbb{E}\int\limits_{0}^{t}\int\limits_{\mathcal{E}}\sum_{j=0}^{n-1}\|c(t_j,X^M(t_j),y) - c(t_j,\tilde{X}_{M,n}^{RE}(t_j),y)\|^p\one_{\mathcal{E}\times(t_j,t_{j+1}]}(y,s)\nu(\rd y)\rd s.
	\end{displaymath}
	Thanks to Lemma \ref{lemma_sol} together with the fact that $\mathbf{1}_{\mathcal{E}\times (t_j,t_{j+1}]}(y,s)=\mathbf{1}_{\mathcal{E}}(y)\cdot\mathbf{1}_{(t_j,t_{j+1}]}(s)$ we obtain
	\begin{eqnarray}
	\label{lemma_C_estimation1}
		&&\mathbb{E}[C_{1,n}^{RE,M}(t)] \leq\mathbb{E}\int\limits_0^t\sum\limits_{j=0}^{n-1}\Biggl(\int\limits_{\mathcal{E}}\|c(s,X^M(s-),y) - c(s,X^M(t_j),y)\|^p\nu(\rd y)\Biggr)\one_{(t_j,t_{j+1}]}(s)\rd s \notag\\
		&&\leq K_{17} \mathbb{E}\int\limits_{0}^t \sum_{j=0}^{n-1}\|X^M(s-) - X^M (t_j)\|^p \one_{(t_j, t_{j+1}]}(s) ds\notag\\
		&&=K_{17} \mathbb{E}\int\limits_{0}^t \sum_{j=0}^{n-1}\|X^M(s) - X^M (t_j)\|^p \one_{[t_j, t_{j+1})}(s) \rd s\notag \\
		&&\leq K_{18} \int\limits_{0}^{t}\sum\limits_{j=0}^{n-1}|s-t_j| \one_{[t_j, t_{j+1})}(s)\rd s \leq K_{19} n^{-1},
	\end{eqnarray}
Again by Lemma \ref{lemma_sol}, we obtain
	\begin{equation}\label{lemma_C_estimation2}
		\begin{split}
			\mathbb{E}\big[C_{2,n}^{RE,M}(t)\big] &\leq  K_{20}\int\limits_{0}^{t}\sum_{j=0}^{n-1}\big(1+\mathbb{E}\|X^M(t_j)\|^p \big)|s-t_j|^{p\varrho_2}\one_{(t_j,t_{j+1}]}(s)\rd s\\
			& \leq K_{21} n^{-p\varrho_2}.
		\end{split}
	\end{equation}
	The following also holds for $t \in [0,T]$
	\begin{equation}\label{lemma_C_estimation3}
	    \mathbb{E}\big[C_{3,n}^{RE,M}(t)\big]
	    \leq L^p \int\limits_{0}^{t}\sup\limits_{0 \leq u \leq s}\mathbb{E}\|X^M(u) - \tilde{X}_{M,n}^{RE}(u)\|^p \ \rd s.  
	\end{equation}
	In view of 	\eqref{lemma_C_estimation1}, \eqref{lemma_C_estimation2}, and \eqref{lemma_C_estimation3} we can see that
	\begin{equation}\label{lemma_C_final}
		\mathbb{E}\|C_n^M(t)\|^p \leq K_{22} (n^{-1} + n^{-p\varrho_2}) + K_{23}\int\limits_{0}^{t}\sup_{0 \leq u \leq s}\mathbb{E}\|X^M(u) - \tilde{X}_{M,n}^{RE}(u)\|^p \  \rd s.
	\end{equation}
 Finally, by \eqref{lemma_AB_RE}, \eqref{lemma_A1_RE}, \eqref{lemma_B_RE_final}, and \eqref{lemma_C_final} the following holds for $t \in [0,T]$
	\begin{displaymath}
		\sup_{0 \leq u \leq t}\mathbb{E}\|X^M(u) - \tilde{X}_{M,n}^{RE}(u)\|^p \leq K_{24} n^{-p\min\{\varrho_1,\varrho_2,1/p\}} + K_{25} \int\limits_{0}^{t}\sup_{0 \leq u \leq s}\mathbb{E}\|X^M(u) - \tilde{X}_{M,n}^{RE}(u)\|^p \rd s,
	\end{displaymath}
	where $K_{24}, K_{25}$ depend only on the parameters of the class $\calf(p, C,D,L,\Delta,\varrho_1, \varrho_2, \nu)$. Moreover, from Lemma \ref{lemma_4} and \eqref{lemma_sol_estimate}, the function $\displaystyle{[0,T]\ni t \mapsto  \sup_{0 \leq u \leq t}\mathbb{E}\|X^M(u) - \tilde{X}_{M,n}^{RE}(u)\|^p}$ is Borel (as a non-decreasing function) and bounded. An application of the Gronwall's lemma yields then
	\begin{displaymath}
		\sup_{0 \leq t \leq T}\mathbb{E}\|X^M(t) - \tilde{X}_{M,n}^{RE}(t)\|^p \leq C_{0} n^{-p\min\{\varrho_1, \varrho_2, 1/p\}},
	\end{displaymath}
	where $C_{0}$ depends only on the parameters of the class $\calf(p, C,D,L,\Delta,\varrho_1, \varrho_2, \nu)$.
	
	In other cases, when some of the coefficients $b,c$ vanish, we use the same proof technique. Note that the $L^p(\Omega)$-regularity of a solution $X^M$ might increase due to Lemma \ref{lemma_sol}, which results in different, usually higher, convergence rates.
\end{proof}
{\noindent\bf Proof of Theorem \ref{upper_bound}.}
 For any $t\in [0,T]$ it holds
\begin{displaymath}
     \Vert X(t) - \tilde{X}_{M,n}^{RE}(t)\Vert_{L^p(\Omega)} 
    \leq  \Vert X(t) - X^M(t)\Vert_{L^p(\Omega)} +
     \Vert X^M(t) - \tilde{X}_{M,n}^{RE}(t)\Vert_{L^p(\Omega)},
\end{displaymath}
and applying Propositions \ref{aux_lem_1}, \ref{aux_lem_2} we get the thesis. \ \ \ $\square$
\section{Lower error bounds and complexity}\label{sec:lower_bounds}
In this section we provide some insight on lower error bounds and complexity bounds of numerically solving \eqref{main_equation}. We consider  the following subclasses of the main class $\calf (p, C,D,L,\Delta,\varrho_1, \varrho_2, \nu)$
\begin{displaymath}
\calg_i(p, C,D,L,\Delta,\varrho_1, \varrho_2, \nu) = \mathcal{A}(D,L) \times \mathcal{B}(C,D,L,\Delta,\varrho_1)\times \mathcal{C}_i(p,D,L,\varrho_2, \nu) \times \mathcal{J}(p,D)
\end{displaymath}
for $i=1,2$, where
\begin{eqnarray*}
    &&\mathcal{C}_1(p,D,L,\varrho_2, \nu)=\{c\in\mathcal{C}(p,D,L,\varrho_2, \nu) \ | \ c(t,x,y)=c(t,x,0) \notag\\ 
    &&\quad\quad\quad\quad\quad\quad\hbox{for all} \ (t,x,y)\in [0,T]\times\mathbb{R}^d\times\mathbb{R}^{d'}\},
\end{eqnarray*}
\begin{eqnarray*}
    &&\mathcal{C}_2(p,D,L,\varrho_2, \nu)=\{c\in\mathcal{C}(p,D,L,\varrho_2, \nu) \ | \ \exists_{\tilde c:[0,T]\times\mathbb{R}^d\mapsto\mathbb{R}^{d\times d'}}:c(t,x,y)=\tilde c(t,x)y\notag\\ 
    &&\quad\quad\quad\quad\quad\quad\hbox{for all} \ (t,x,y)\in [0,T]\times\mathbb{R}^d\times\mathbb{R}^{d'}\}.
\end{eqnarray*}
Note that $\mathcal{C}_1(p,D,L,\varrho_2, \nu)\subset C([0,T]\times\mathbb{R}^d\times\mathbb{R}^{d'};\mathbb{R}^d)$, see Remark \ref{class_c1}. For the class $\mathcal{G}_2$ we additionally impose the following assumption on the L\'{e}vy measure 
    \begin{displaymath}
     {\rm (D)} \quad\quad\quad   \displaystyle{\kappa_p:=\Biggl(\int\limits_{\mathcal{E}}\|y\|^p\nu(\rd y)\Biggr)^{1/p}<+\infty},
    \end{displaymath}
which assures that $\mathcal{C}_2$ is non-empty, see Remark \ref{sub_c_2}. Moreover, for all $t\in [0,T]$
\begin{displaymath}
	\int\limits_0^t\int\limits_{\mathcal{E}}c(s,X(s-),y)N(\rd y,\rd s)=\left\{ \begin{array}{ll}
	\displaystyle{
	 \int\limits_0^t c(s,X(s-),0) \rd N(s)}, \ \hbox{if} \ c\in\mathcal{C}_1,\\
	 \displaystyle{
	 \int\limits_0^t \tilde c(s,X(s-))\rd L(s), \ \hbox{if} \ c\in\mathcal{C}_2.
	 }
	\end{array} \right.
\end{displaymath}
Both classes $\mathcal{G}_1$, $\mathcal{G}_2$ are important from a point of view of possible applications in finance, see, for example, \cite{carter}, \cite{PBL}.

We consider a class $\Phi_{M,n}$ of algorithms $\bar X_{M,n}$ that are parametrized by the pair $(M,n)$, where $n\in\mathbb{N}$ is a discretization parameter while   $M\in\mathbb{N}$ is a truncation dimension parameter. In a subclass $\mathcal{G}\in\{\calg_1,\calg_2\}$ we assume that any $\bar X_{M,n}$ uses only finite dimensional discrete information about the coefficients $(a,b,c)$, the Wiener process $W,$ and the process  $Z\in\{N,L\}$. Namely, the vector of information used by $\bar X_{M,n}$ is of the following form 
\begin{eqnarray*}
    &&\mathcal{N}_{M,n}(a,b,c,\eta,W,Z) = \Bigl[a(\theta_0,y_0), a(\theta_1,y_1),\ldots, a(\theta_{k_1-1},y_{k_1-1}),\notag\\
    &&\quad\quad\quad\quad\quad\quad\quad b^M(t_0,z_0), b^M(t_1,z_1),\ldots, b^M(t_{k_1-1},z_{k_1-1}),\notag\\
    &&\quad\quad\quad\quad\quad\quad\quad \bar c(u_0,v_0), \bar c(u_1,v_1),\ldots, \bar c(u_{k_1-1},v_{k_1-1}),\notag\\
    &&\quad\quad\quad\quad\quad\quad\quad W^M(s_0), W^M(s_1), \ldots, W^M(s_{k_2-1}),\notag\\
    &&\quad\quad\quad\quad\quad\quad\quad
    Z(q_0),Z(q_1),\ldots, Z(q_{k_3-1}),\eta\Bigr],
\end{eqnarray*}
where $\bar c(t,x)=c(t,x,0)$ (if $Z=N$) or $\bar c(t,x)=\tilde c(t,x)$ (if $Z=L$), and  $W^M=[W_1,W_2,\ldots,W_M]^T$ is the $M$-dimensional Wiener process. We assume that $k_i\in\mathbb{N}$, for $i=1,2,3$, are given  and such that
\begin{equation}
\label{upper_kn}
 \max\limits_{1\leq i\leq 3}k_i= O(n),
\end{equation}
$[\theta_0,\theta_1,\ldots,\theta_{k_1-1}]^T$ is a $[0,T]^{k_1}$-valued random vector on $(\Omega,\Sigma,\mathbb{P})$, such that the $\sigma$-fields $\sigma(\theta_0,\theta_1,\ldots,\theta_{k_1-1})$ and $\Sigma_{\infty}$ are independent. Furthermore, $t_0,t_1,\ldots,t_{k_1-1},\ s_0,s_1,\ldots,s_{k_2-1}$,\\ $u_0,u_1,\ldots,u_{k_1-1}$,$q_0,q_1,\ldots,q_{k_3-1}\in [0,T]$ are given discretization points such that $t_i\neq t_j$, $u_i\neq u_j$, $s_i\neq s_j$, $q_i\neq q_j$ for $i\neq j$. The evaluation points $y_j, z_j, u_j$ for the spatial variables $y_j$, $z_j$, $v_j$ of $a(\cdot, y), b^M(\cdot, z)$, and $\bar c(\cdot, v)$ can be given in adaptive way with respect to $(a,b,c,\eta)$, $W$, and $Z$. It means that there exist Borel measurable functions  $\psi_j$, $j=0,1,\ldots,k_1-1$, such that the successive points $y_j,z_j$ are computed in the following way
\begin{displaymath}
	(y_0,z_0,v_0)=\psi_0\Bigl(W^M(s_0),\ldots, W^M(s_{k_2-1}),Z(q_0),\ldots,Z(q_{k_3-1}),\eta\Bigr), 
\end{displaymath}
and for $j=1,2,\ldots, k_1-1$ 
\begin{eqnarray*}
	(y_j,z_j,v_j)&=&\psi_j\Bigl(a(\theta_0,y_0),  a(\theta_1,y_1),\ldots,  a(\theta_{j-1},y_{j-1}),\notag\\
	&& \ \ \quad b^M(t_0,z_0), b^M(t_1,z_1),\ldots, b^M(t_{j-1},z_{j-1}),\notag\\
	&& \ \ \quad \bar c(u_0,v_0),\bar c(u_1,v_1),\ldots,\bar c(u_{j-1},v_{j-1}),\notag\\
	&& \ \ \quad W(s_0), W(s_1),\ldots, W(s_{k_2-1}),\notag\\
	&& \ \ \quad Z(q_0),Z(q_1),\ldots, Z(q_{k_3-1}), \eta\Bigr). 
\end{eqnarray*}
By the (informational)  cost of computing the information $\mathcal{N}_{M,n}$ we mean the total  number of scalar (finite dimensional) evaluations of $(a,b,c,\eta)$, $W,$ and $Z$. Hence, for all $(a,b,c,\eta)\in\mathcal{G}$ and for all trajectories of $W$ and $Z\in\{N,L\}$ it is equal to 
\begin{equation*}
2dk_1+M(dk_1  + k_2)+k_3 +d 
\end{equation*}
if $(Z,\mathcal{G})=(N,\mathcal{G}_1)$ and 
\begin{equation*} 
d(1+d')k_1+M(dk_1 + k_2)+k_3d'+d
\end{equation*}
if $(Z,\mathcal{G})=(L,\mathcal{G}_2)$. However, by \eqref{upper_kn} in both cases the cost is  $O(Mn)$.

An algorithm $\bar X_{M,n}\in\Phi_{M,n}$, using $\mathcal{N}_{M,n}$,  that approximates $X(T)$ is given by
\begin{equation}\label{phi_algorithm_def}
    \bar X_{M,n}(a,b,c,\eta, W,Z) = \phi_{M,n}\big(\mathcal{N}_{M,n}(a,b,c,\eta, W,Z)\big),
\end{equation}
for some Borel measurable function
\begin{equation*}
    \phi_{M,n}:\mathbb{R}^{ d\times k_1+ d\times (Mk_1) + d\times k_1+M\times k_2+1\times k_3+d\times 1}\mapsto \mathbb{R}^d
\end{equation*}
when $Z=N$ and 
\begin{equation*}
    \phi_{M,n}:\mathbb{R}^{d\times k_1+ d\times (Mk_1) + d\times (d' k_1)+M\times k_2+d'\times k_3+d\times 1}\mapsto \mathbb{R}^d
\end{equation*}
if $Z=L$. The error of $\bar X_{M,n}\in\Phi_{M,n}$ for a fixed $(a,b,c,\eta)\in\mathcal{G}$, where $(Z,\mathcal{G})\in \{(N,\mathcal{G}_1),(L,\mathcal{G}_2)\}$, is defined as
\begin{equation*}
    e^{(p)}\Bigl(\bar X_{M,n},(a,b,c,\eta)\Bigr)=\Bigl(\mathbb{E}\|X(a,b,c,\eta)(T)-\bar X_{M,n}(a,b,c,\eta,W,Z)\|^p\Bigr)^{1/p}.
\end{equation*}
The worst-case error of $\bar X_{M,n}\in\Phi_{M,n}$ in $\mathcal{G}$ is given by 
\begin{equation*}
    e^{(p)}(\bar X_{M,n},\mathcal{G})=\sup\limits_{(a,b,c,\eta)\in \mathcal{G}}e^{(p)}\Bigl(\bar X_{M,n},(a,b,c,\eta)\Bigr),
\end{equation*}
see \cite{TWW88}. For $\varepsilon\in (0,+\infty)$ we define the $\varepsilon$-complexity in $\mathcal{G}\in\{\mathcal{G}_1,\mathcal{G}_2\}$ as follows
\begin{eqnarray*}
    &&{\rm comp}(\varepsilon,\mathcal{G})=\inf\Bigl\{ nM   \ | \  M,n\in\mathbb{N} \ \hbox{are such that}\notag\\
    && \ \exists_{\phi_{M,n},\mathcal{N}_{M,n}} \ \hbox{with} \ \max\limits_{1\leq i\leq 3}k_i=O(n), \ \max\limits_{i=1,2}k_i=\Omega(n) \ \hbox{and} \ \ e^{(2)}(\bar X_{M,n},\mathcal{G})\leq\varepsilon, \notag\\
    &&\hbox{where} \ \bar X_{M,n}=\phi_{M,n}\circ\mathcal{N}_{M,n} \Bigr\}.
\end{eqnarray*}
Note that we consider complexity  only for the worst-case error measured in  $L^2(\Omega)$-norm (i.e., for $p=2$). Moreover, we narrow our attention to algorithms $\bar X_{M,n}$ for which $k_1+k_2=\Omega(n)$, which in turn implies that the cost of any such algorithm is $\Omega(Mn)$. Such assumption is not too restrictive and is often satisfied by algorithm used in practice, for example, for $\bar X^{RE}_{M,n}$.

In order to establish upper bound on the complexity we need the following corollary that directly follows from  Theorem \ref{upper_bound}.
\begin{corollary}
\label{upper_bound_12}
\noindent
\begin{itemize}
    \item [(i)] 	 There exists a positive constant $K$, depending only on the parameters of the class $\mathcal{G}_1(p,C,D,L,\Delta,\varrho_1, \varrho_2, \nu)$, such that for every $M,n\in\mathbb{N}$ and $(a,b,c, \eta)\in\mathcal{G}_1(p, C,D,L,\Delta,\varrho_1, \varrho_2, \nu)$ it holds
	\begin{equation*}
	    \|X(a,b,c,\eta)(T)-\bar X^{RE}_{M,n}(a,b,c,\eta,W,N)\|_{L^p(\Omega)}\leq K\Bigl( n^{-\min\{\varrho_1, \varrho_2, 1/p\}} +  \delta(M)\Bigr).
	\end{equation*}
	\item [(ii)] Let $\kappa_p<+\infty$. There exists a positive constant $K$, depending only on the parameters of the class $\mathcal{G}_2(p,C,D,L,\Delta,\varrho_1, \varrho_2, \nu)$, such that for every $M,n\in\mathbb{N}$ and $(a,b,c, \eta)\in\mathcal{G}_2(p, C,D,L,\Delta,\varrho_1, \varrho_2, \nu)$ it holds
	\begin{equation*}
	    \|X(a,b,c,\eta)(T)-\bar X^{RE}_{M,n}(a,b,c,\eta,W,L)\|_{L^p(\Omega)}\leq K\Bigl( n^{-\min\{\varrho_1, \varrho_2, 1/p\}} +  \delta(M)\Bigr).
	\end{equation*}
\end{itemize}
For the both classes the (informational) cost of $\bar X^{RE}_{M,n}$ is $\Theta (Mn)$.
\end{corollary}
In the following part of the section we deal with  suitable lower error bounds in $\mathcal{G}_i$, $i=1,2$. For a better clarity the proof is divided into number of auxiliary lemmas stated below.
\begin{lemma} 
\label{low_a_prop1}
(Lower error bound for the Lebesgue integration) For $i=1,2$ and for any algorithm $\bar X_{M,n} \in \Phi_{M,n}$ it holds that
    \begin{equation*}
        e^{(p)}(\bar X_{M,n},\mathcal{G}_{i})  = \Omega(n^{-1/2}).
    \end{equation*}
\end{lemma}
\begin{proof}
We consider the following class
\begin{equation*}
    \mathcal{M}_{1} : =\mathcal{\bar A} \times \{0\} \times \{0\} \times\{0\}
\end{equation*}
with 
\begin{equation*}
    \mathcal{\bar A}=\{ a\in \mathcal{A}(D, L) \ | \ a(t,x) = a(t,0) \text{ for all } t \in [0,T], x \in \mathbb{R}^d \}.
\end{equation*}
For $(a,b,c,\eta) \in \mathcal{M}_{1}$ we have $\displaystyle{X(a,b,c,\eta)(T)=\int_{0}^{T}a(t,0)\rd t}$. Since $\mathcal{M}_{1}\subset\mathcal{G}_{i}$ for $i=1,2$, $k_1=O(n)$, by Theorem 4.2.1 from Chapter 11 in \cite{TWW88} we obtain the thesis.
\end{proof}
In the sequel we will make use of the following lemma.
\begin{lemma}
\label{low_b_lem1} Let $Z\in\{N,L\}$. There exists $M_0\in\mathbb{N}$, depending only on the parameters of the class $\calf(p, C,D,L,\Delta,\varrho_1, \varrho_2, \nu)$, such that for all $M\geq M_0$ and any $\sigma(\mathcal{H}_M\cup\Sigma^Z_{\infty})$-measurable random vector  $Y:\Omega\mapsto\mathbb{R}^d$ it holds
 \begin{equation}
 \label{low_b11}
     \sup\limits_{b\in\mathcal{B}_0(C,D,L,\Delta,\varrho_1)}\mathbb{E}\|\mathcal{I}(b)-Y\|^2\geq\sup\limits_{b\in\mathcal{B}_0(C,D,L,\Delta,\varrho_1)}\mathbb{E}\|\mathcal{I}(b)-\mathcal{I}(b^M)\|^2\geq C^{2}T(\delta(M))^2,
 \end{equation}
 where $\displaystyle{\mathcal{I}(b)=\int\limits_0^T b(t,0)\rd W(t)}$, and
\begin{eqnarray*}
	&&\mathcal{B}_0(C,D,L,\Delta,\varrho_1)=\{ b\in \mathcal{B}(C,D,L,\Delta,\varrho_1) \ | \ b(t,x) = b(t,0)\notag\\
	&&\quad\quad\quad\quad\quad\quad\text{ for all } t \in [0,T], x \in \mathbb{R}^d \}.
\end{eqnarray*}
\end{lemma}
\begin{proof} 
Firstly, note that for all $M\in\mathbb{N}$
we have that $\sup\limits_{b\in\mathcal{B}_0(C,D,L,\Delta,\varrho_1)}\mathbb{E}\|\mathcal{I}(b)-\mathcal{I}(b^M)\|^2<~+\infty$.

Let $M\in\mathbb{N}$ and let us consider any $\sigma(\mathcal{H}_M\cup\Sigma^Z_{\infty})$-measurable $Y:\Omega\mapsto\mathbb{R}^d$. If $\mathbb{E}\|Y\|^2=+\infty$ then $\displaystyle{ \sup\limits_{b\in\mathcal{B}_0(C,D,L,\Delta,\varrho_1)}\mathbb{E}\|\mathcal{I}(b)-Y\|^2=+\infty}$ and the first inequality in \eqref{low_b11} is obvious. If $\mathbb{E}\|Y\|^2<+\infty$ then by the projection property of conditional expectation we get for all $b\in\mathcal{B}_0(C,D,L,\Delta,\varrho_1)$ that
\begin{equation*}
    \mathbb{E}\|\mathcal{I}(b)-Y\|^2\geq\mathbb{E}\|\mathcal{I}(b)-\mathbb{E}(\mathcal{I}(b) \ | \ \sigma(\mathcal{H}_M\cup\Sigma^Z_{\infty}))\|^2.
\end{equation*}
Since $\displaystyle{\mathcal{I}(b^M)=\sum\limits_{j=1}^M\int\limits_0^T b^{(j)}(t,0)\text{d}W_j(t)}$ is $\mathcal{H}_M$-measurable, $\displaystyle{\mathcal{I}(b-b^M)=\sum\limits_{j = M+1}^{+\infty}\int\limits_0^T b^{(j)}(t,0)\text{d}W_j(t)}$ is $\mathcal{H}^+_M$-measurable, and $\mathbb{E}(\mathcal{I}(b-b^M) \ | \ \sigma(\mathcal{H}_M\cup\Sigma^Z_{\infty}))=0$, we get
\begin{equation*}
    \mathbb{E}\|\mathcal{I}(b)-\mathbb{E}(\mathcal{I}(b) \ | \ \sigma(\mathcal{H}_M\cup\Sigma^Z_{\infty}))\|^2\notag\\
    =\mathbb{E}\|\mathcal{I}(b-b^M)\|^2.
\end{equation*}
This ends the proof of the first inequality in \eqref{low_b11}.

Let us consider the function 
\begin{equation}
\label{tilde_bN}
    \tilde b_M = (\tilde b^{(1)}_M,\ldots,\tilde b^{(M)}_M, \tilde b^{(M+1)}_M,\ldots), 
\end{equation}
where $\tilde b_M^{(j)}=[0,\ldots,0]^T$ for all $j\neq M+1$ and $\tilde b_M^{(M+1)}=[C\delta(M),0,\ldots,0]^T$. Note that there exists $M_0$ such that for all $M\geq M_0$ we have $\tilde b_M\in\mathcal{B}_0(C,D,L,\Delta,\varrho_1)$. Moreover,
\begin{equation*}
    \mathbb{E}\|\mathcal{I}(\tilde b_M)-\mathcal{I}(P_M\tilde b_M)\|^2= C^{2}T(\delta(M))^2,
\end{equation*}
which ends the proof. 
\end{proof}
\begin{lemma}
\label{ito_lower_bounds} 
(Lower error bound for the stochastic It\^o integration) Let $i=1,2$. There exist positive constants $C_0$, $n_0,M_0 \in \mathbb{N}$, depending only on the parameters of the class $\mathcal{G}_i(p,C,D,L,\Delta,\varrho_1,\varrho_2,\nu)$, such that for every $n\geq n_0$, $M \geq M_0$ and $\bar X_{M,n} \in \Phi_{M,n}$ we have the following lower bound
	\begin{equation}
	\label{lowerbound_main_th}
		e^{(p)}(\bar X_{M,n}, \mathcal{G}_i) \geq C_0 \max\{n^{-\varrho_1}, \delta (M)\}.
	\end{equation}
\end{lemma}
\begin{proof}
We split the proof into two parts corresponding to different components of the lower bound \eqref{lowerbound_main_th}.
Firstly, we define the following class
\begin{eqnarray*}
	&&\mathcal{M}_2 =  \{0\} \times \{ b\in \mathcal{B}(C,D,L,\Delta,\varrho_1) \ | \ b(t,x) = b(t,0), b^{(j)}_i(t,x)=0\notag\\
	&&\quad\quad\quad\quad\quad\quad\text{ for all } \ i+j\geq 3,  \   t \in [0,T], x \in \mathbb{R}^d \} \times \{0\} \times \{0\},
\end{eqnarray*}
where $\mathcal{M}_2\subset\mathcal{G}_i$ for $i=1,2$. For any $(a,b,c,\eta)\in \mathcal{M}_2$ we have
\begin{equation*}
    X(a,b,c,\eta)(T)=\int\limits_0^T b^{(1)}(t,0)\text{d}W_1(t)=\Biggl[\int\limits_0^T b^{(1)}_{1}(t,0)\text{d}W_1(t),0,\ldots,0\Biggr]^T.
\end{equation*}
Hence, from Proposition 5.1. (i) in \cite{PMPP14} and by the fact that $k_1=O(n)$ we get that  for any $\bar X_{M,n} \in \Phi_{M,n}$ that
\begin{equation}
\label{lowerbound_varrho1}
    e^{(p)}(\bar X_{M,n}, \mathcal{M}_2) =\Omega(n^{-\varrho_1}).
\end{equation}
(Note that by the results of \cite{Hein1} the lower bound \eqref{lowerbound_varrho1} holds also in case when the evaluation points for $W^N$ are chosen in an adaptive way.)

In order to establish a lower bound \eqref{lowerbound_main_th} dependent on $M,$ we consider the following class
\begin{equation*}
	\mathcal{M}_3 =  \{0\} \times \mathcal{B}_0(C,D,L,\Delta,\varrho_1)\times \{0\} \times \{0\},
\end{equation*}
where $\mathcal{M}_3\subset\mathcal{G}_i$ for $i=1,2$. For every $(a,b,c, \eta) \in\mathcal{M}_3$ we have 
\begin{equation*}
	X(a,b,c,\eta)(T) = \mathcal{I}(b)=\int\limits_{0}^{T}b(t,0)\text{d}W(t).
\end{equation*}
 By Lemma \ref{low_b_lem1} we have that there exists $M_0\in\mathbb{N}$ such that for all $M\geq M_0$, $n\in\mathbb{N}$, and any algorithm $\bar X_{M,n} \in \Phi_{M,n}$ the following holds
\begin{eqnarray}
\label{lowerbound_deltaN}
&&e^{(p)}(\bar X_{M,n}, \mathcal{M}_3) \geq \sup\limits_{b \in \mathcal{B}_0(C,D,L,\Delta,\varrho_1)}\Bigl(\mathbb{E}\|X(0,b,0,0)(T) - \bar X_{M,n}(0,b,0,0,W,Z)\|^2\Bigr)^{1/2}\notag\\ 
&&\geq CT^{1/2}\delta(M),
\end{eqnarray}
since by \eqref{phi_algorithm_def} we have that $\sigma(\bar X_{M,n}(0,b,0,0,W,Z))\subset\sigma(\mathcal{H}_M\cup\Sigma^Z_{\infty})$ for  $Z\in \{N,L\}$. Since $\mathcal{M}_i\subset \mathcal{G}_j, \ i=2,3, \ j =1,2,$ by  \eqref{lowerbound_varrho1} and \eqref{lowerbound_deltaN} we get 
\begin{equation*}
	e^{(p)}(\bar X_{M,n},\mathcal{G}_i) \geq\max\{ e^{(p)}(\bar X_{M,n}, \mathcal{M}_2),  e^{(p)}(\bar X_{M,n}, \mathcal{M}_3)\} = \Omega(\max\{n ^{-\varrho_1}, \delta(M)\}).
\end{equation*}
This completes the proof.
\end{proof}
\begin{lemma}
\label{lower_bound_c1_class}
(Lower error bound in the class $\mathcal{G}_1$ for the stochastic integration wrt Poisson random measure) There exist positive constants $C_1$ and $n_0 \in \mathbb{N}$, depending only on the parameters of the class $\mathcal{G}_1(p,C,D,L,\Delta,\varrho_1, \varrho_2, \nu)$, such that for every $n \geq n_0$, $M\in\mathbb{N}$ and for every algorithm $\bar X_{M,n} \in \Phi_{M,n}$ it holds
    \begin{equation*}
        e^{(p)}(\bar X_{M,n}, \mathcal{G}_1) \geq C_1 n^{-\varrho_2}.
    \end{equation*}
\end{lemma}
\begin{proof} Let us consider the following class
\begin{equation*}
    \mathcal{M}_4=\{0\} \times \{0\} \times \mathcal{\bar C}_1 \times \{0\},
\end{equation*}
where
\begin{displaymath}
	\mathcal{\bar C}_1=\{ c\in \mathcal{C}(p,D,L,\varrho_2, \nu) \ | \ c(t,x,y) = c(t,0,0) \text{ for all } (t,x,y) \in [0,T]\times\mathbb{R}^d\times\mathbb{R}^{d'} \},
\end{displaymath}
and $\mathcal{M}_4\subset\mathcal{G}_1$. If we consider the input vector $ (a,b,c,\eta) \in \mathcal{M}_4$ then it holds
\begin{equation*}
X(a,b,c,\eta)=\mathcal{I}(c):=\int\limits_{0}^{T} \int\limits_{\mathcal{E}} c(t,0,0) N(\rd y, \rd t) =\int\limits_{0}^{T}c(t,0,0) \rd N(t).
\end{equation*}
Let $\bar X_{M,n} \in \Phi_{M,n}.$ By using suitably chosen bump functions we can construct two mappings $c_1,c_2\in\mathcal{\bar C}_1$ such that $c_1(t,x,y)=[c_{1,1}(t),0,\ldots,0]^T$, $c_2(t,x,y)=[c_{2,1}(t),0,\ldots,0]^T$, $\displaystyle{\Bigl|\int\limits_0^T (c_{1,1}(t)-c_{2,1}(t))dt \Bigr|=\Omega(n^{-\varrho_2})}$ and  $c_{1,1}(u_j)=0=c_{2,1}(u_j)$ for $j=0,1,\ldots,k_1-1$, where $k_1=O(n)$. (The construction of such mappings is well-known and widely used in the literature when establishing lower bounds, see, for example, \cite{NOV}). Then we have that $\bar X_{M,n}(0,0,c_1,0, W, N) = \bar X_{M,n}(0,0,c_2,0, W, N)$, and by Lemma 6.1 (i) in \cite{JDPP} we get
$$
e^{(p)}(\bar X_{M,n}, \mathcal{M}_4) \geq e^{(2)}(\bar X_{M,n}, \mathcal{M}_4) \geq \frac{1}{2}\left(\mathbb{E}\|\mathcal{I}(c_{1}) - \mathcal{I}(c_{2})\|^2\right)^{1/2}.
$$
$$
    \geq \frac{1}{2}\left[ \lambda \int\limits_0^T \big(c_{1, 1}(t)-c_{2, 1}(t)\big)^2 \rd t + \lambda^2 \biggr(\int\limits_0^T \big(c_{1, 1}(t)-c_{2,1}(t)\big)\rd t\biggr)^2\right]^{1/2}
$$
$$
    \geq K\Bigl| \int\limits_0^T \big(c_{1, 1}(t)-c_{2,1}(t)\big)\rd t\Bigl|=\Omega(n^{-\varrho_2}).
$$
This ends the proof.
\end{proof}
The analogous lower bound can be obtained in class $\mathcal{G}_2$.
\begin{lemma}
\label{lb_poiss_2}
(Lower error bound in the class $\mathcal{G}_2$ for the stochastic integration wrt Poisson random measure)     Let $\kappa_p<+\infty$.     There exist positive constants $C_1$ and $n_0 \in \mathbb{N}$, depending only on the parameters of the class $\mathcal{G}_2(p,C,D,L,\Delta,\varrho_1, \varrho_2, \nu)$, such that for every $n \geq n_0$, $M\in\mathbb{N}$ and for every algorithm $\bar X_{M,n} \in \Phi_{M,n}$ we have
    \begin{equation*}
        e^{(p)}(\bar X_{M,n}, \mathcal{G}_2) \geq C_1 n^{-\varrho_2}.
    \end{equation*}
\end{lemma}
\begin{proof} Let
\begin{equation*}
    \mathcal{M}_5=\{0\} \times \{0\} \times \mathcal{\bar C}_2 \times \{0\},
\end{equation*}
where
\begin{displaymath}
\mathcal{\bar C}_2=\{ c\in \mathcal{C}(p,D,L,\varrho_2, \nu) \ | \ \exists_{\tilde c:[0,T]\times\mathbb{R}^d\mapsto\mathbb{R}^{d\times d'}}: \ c(t,x,y) = \tilde c(t,0)y \ \hbox{for all} \ (t,x,y)\in [0,T]\times\mathbb{R}^d\times\mathbb{R}^{d'} \}.
\end{displaymath}
We have $\mathcal{M}_5\subset\mathcal{G}_2$. Furthermore, for every $ (a,b,c,\eta) \in\mathcal{M}_5$ it holds
\begin{displaymath}
X(a,b,c,\eta)=\mathcal{I}(c):=\int\limits_{0}^{T} \int\limits_{\mathcal{E}} \tilde c(t,0)y N(\rd y, \rd t) = 
\biggr[ \sum_{j=1}^{d'}\int\limits_0^t \tilde c_{ij}(t,0)\rd L_j (t)\biggr]_{i=1,\ldots, d} = \int\limits_{0}^{T}\tilde c(t,0) \rd L(t),
\end{displaymath}
where the last integration is with respect to $d'$--dimensional compound Poisson process $L(t) = [L_1(t), \ldots, L_{d'}(t)]^T, \ t \in [0,T],$ and $\xi_k = [\xi_k^{(1)}, \ldots, \xi_k^{(d')}]^T, \ k\geq 1,$ are iid $\mathcal{E}$-valued random variables with the common distribution $\nu(\rd y)/\lambda.$ By $\kappa_p<+\infty$ it holds for all $k\in\mathbb{N}$, and $l=1,2,\ldots,d'$ that $(\mathbb{E}|\xi_k^{(l)}|^2)^{1/2}=(\mathbb{E}|\xi_1^{(l)}|^2)^{1/2} \leq (\mathbb{E}\|\xi_1\|^2)^{1/2}\leq  (\mathbb{E}\|\xi_1\|^p)^{1/p}< +\infty.$ Since $\mathbb{P}(\xi_1=0)=0$ we have that $\mathbb{E}\|\xi_1\|^2>0$. Therefore, there exists $j_0\in\{1,2,\ldots,d'\}$ such that $\mathbb{E}|\xi_1^{(j_0)}|^2>0$, although it might happen that $\mathbb{E}\xi_1^{(j_0)}=0.$ 

The next part follows analogous proof technique as in the proof of Lemma \ref{lower_bound_c1_class}. Let $\bar X_{M,n} \in \Phi_{M,n}.$ We construct two mappings $c_1,c_2\in\mathcal{\bar C}_2$ such that $c_1(t,x,y)=\tilde c_1(t)y$, $c_2(t,x,y)=\tilde c_2(t)y$, $\tilde c_{2,ij}\equiv 0\equiv \tilde c_{1,ij}$ whenever $(i,j)\neq (1,j_0)$, $\displaystyle{\Bigl|\int\limits_0^T \big(\tilde c_{1,1j_0}(t)-\tilde c_{2,1j_0}(t)\big)dt \Bigr|=\Omega(n^{-\varrho_2})}$ and  $\tilde c_{1,1j_0}(u_j)=0=\tilde c_{2,1j_0}(u_j)$ for $j=0,1,\ldots,k_1-1$ with $k_1=O(n)$. Consequently, we obtain $\bar X_{M,n}(0,0,c_1,0, W, L) = \bar X_{M,n}(0,0,c_2,0, W, L)$ and
\begin{displaymath}
    e^{(p)}(\bar X_{M,n}, \mathcal{M}_5) \geq \frac{1}{2}\left(\mathbb{E}\big\|\mathcal{I}(c_{1}) - \mathcal{I}(c_{2})\big\|^2\right)^{1/2} =
    \frac{1}{2}\biggr( \mathbb{E} \left[\int_{0}^{T} \big(\tilde c_{1,1 j_0}(t)-\tilde c_{2,1 j_0}(t)\big) \rd L_{j_0}(t)\right]^2\biggr)^{1/2}.
\end{displaymath}
For the notational brevity we denote  $\bar{c}_{j_0}:=\tilde c_{1,1 j_0}(t)-\tilde c_{2,1 j_0}$. By the martingale and isometry properties for stochastic integral driven by the c\'adl\'ag martingale $\displaystyle{\Bigl(L_{j_0}(t)-\lambda t \mathbb{E}\xi_1^{(j_0)},\Sigma_t\Bigr)_{t\in [0,T]}}$ (see, for example, Theorem 88, page 53 in \cite{situ})
\begin{displaymath}
       \biggr( e^{(2)}(\bar X_{M,n}, \mathcal{M}_5) \biggr)^{2} \geq
        \frac{1}{4}\mathbb{E}\biggr( \int_{0}^{T} \bar{c}_{j_{0}}(t)\Bigl(\rd L_{j_{0}}(t) -\lambda \mathbb{E} \xi_{1}^{(j_{0})}\rd t\Bigr)         + \lambda \mathbb{E} \xi_{1}^{(j_{0})}\int_{0}^{T}\bar{c}_{j_{0}}(t) \rd t \biggr)^{2}
\end{displaymath}
\begin{eqnarray*}
  &&=   \frac{\lambda}{4}\mathbb{E}|\xi_{1}^{(j_{0})}|^{2}\cdot  \int_{0}^{T} (\bar{c}_{j_{0}}(t))^{2}\rd t 
    + \frac{\lambda^{2}}{4} \Bigl(\mathbb{E}(\xi_{1}^{(j_{0})})\Bigr)^2 \biggr(\int_{0}^{T} \bar{c}_{j_{0}}(t)\rd t\biggr)^{2} \notag\\
   && \geq 
    \frac{\lambda}{4T} \mathbb{E}|\xi_{1}^{(j_{0})}|^{2} \cdot\Biggr(\int_{0}^{T} \bar{c}_{j_{0}}(t)\rd t\Biggl)^{2},
\end{eqnarray*}
which results in the following inequality
\begin{displaymath}
e^{(p)}(\bar X_{M,n}, \mathcal{M}_5) \geq e^{(2)}(\bar X_{M,n}, \mathcal{M}_5) \geq \frac{1}{2} \ \sqrt{\frac{\lambda}{T}} \ \| \xi_{1}^{(j_{0})}\|_{L^{2}(\Omega)} \ \biggr| \int_{0}^{T} \bar{c}_{j_{0}}(t) \rd t \biggr|=\Omega(n^{-\varrho_2}),
\end{displaymath}
and the proof is completed.
\end{proof}
\begin{remark}
    Since $\mathcal{\bar C}_1\cap\mathcal{\bar C}_2=\{0\}$, we have to show lower bounds in Lemmas \ref{lower_bound_c1_class}, \ref{lb_poiss_2} separately for each class $\mathcal{G}_1$ and $\mathcal{G}_2$. Furthermore, see also \cite{JDPP} where the authors established lower error bounds for approximate stochastic integration wrt homogeneous Poisson process but in different class of integrands and in the so-called asymptotic setting.
\end{remark}
From Lemmas \ref{low_a_prop1}-\ref{lb_poiss_2} we obtain the following result.
\begin{theorem}
\label{lower_b_g12}
\noindent
\begin{itemize}
    \item [(i)]  There exist positive constants $\hat C,n_0,M_0$, depending only on the parameters of the class $\mathcal{G}_1(p,C,D,L,\Delta,\varrho_1, \varrho_2, \nu)$, such that for all $n\geq n_0,M\geq M_0$ and for every method $\bar X_{M,n}\in\Phi_{M,n}$ it holds
    \begin{equation*}
        e^{(p)}(\bar X_{M,n},\mathcal{G}_1)\geq \hat C(n^{-\min\{\varrho_1,\varrho_2,1/2\}}+\delta(M)).
    \end{equation*}
    \item [(ii)]   Let $\kappa_p<+\infty$. There exist positive constants $\hat C,n_0,M_0$, depending only on the parameters of the class $\mathcal{G}_2(p,C,D,L,\Delta,\varrho_1, \varrho_2, \nu)$, such that for all $n\geq n_0,M\geq M_0$ and for every method $\bar X_{M,n}\in\Phi_{M,n}$ it holds
    \begin{equation*}
        e^{(p)}(\bar X_{M,n},\mathcal{G}_2)\geq \hat C(n^{-\min\{\varrho_1,\varrho_2,1/2\}}+\delta(M)).
    \end{equation*}
\end{itemize}
\end{theorem}
\subsection{Complexity bounds}
We are ready to establish the optimality of previously defined Euler algorithm \eqref{main_scheme} in the class $\mathcal{G}_i, \ i=1,2.$ and for $p=2$.
\begin{theorem} 
\label{thm_complex}
Let $\gamma=\min\{\varrho_1,\varrho_2,1/2\}$.
\begin{itemize}
    \item [(i)] There exist positive constants $C_1,C_2,C_3,C_4,\varepsilon_0$, depending only on the parameters of the class $\mathcal{G}_1(2,C,D,L,\Delta,\varrho_1, \varrho_2, \nu)$, such that for all $\varepsilon\in (0,\varepsilon_0)$ the following holds
    \begin{equation*}
        C_1(1/\varepsilon)^{1/\gamma}\delta^{-1}(\varepsilon/C_2)\leq {\rm comp}(\varepsilon,\mathcal{G}_1)\leq C_3(1/\varepsilon)^{1/\gamma}\delta^{-1}(\varepsilon/C_4).
    \end{equation*}
    \item [(ii)]  Let $\kappa_2<+\infty$. There exist positive constants $C_1,C_2,C_3,C_4,\varepsilon_0$, depending only on the parameters of the class $\mathcal{G}_2(2,C,D,L,\Delta,\varrho_1, \varrho_2, \nu)$, such that for all $\varepsilon\in (0,\varepsilon_0)$ the following holds
    \begin{equation*}
        C_1(1/\varepsilon)^{1/\gamma}\delta^{-1}(\varepsilon/C_2)\leq {\rm comp}(\varepsilon,\mathcal{G}_2)\leq C_3(1/\varepsilon)^{1/\gamma}\delta^{-1}(\varepsilon/C_4).
    \end{equation*}
\end{itemize}
\end{theorem}
\begin{proof}
    Let us define
    \begin{equation*}
        U(\varepsilon)=\inf\Bigl\{ M n \ | \ M,n  \ \hbox{are such that} \ n^{-\gamma}+\delta(M)\leq \varepsilon\Bigr\}.
    \end{equation*}
    By Corollary \ref{upper_bound_12} we have that  for sufficiently small $\varepsilon>0$
    \begin{equation*}
        {\rm comp}(\varepsilon,\mathcal{G}_i)\leq  U(\varepsilon/K)
    \end{equation*}
    for $i=1,2$. To bound $U(\varepsilon/K)$ from above it is sufficient to take the value of $nM$ with the minimal $n,M$ such that
    \begin{equation*}
        Kn^{-\gamma}\leq \varepsilon/2, \ K\delta(M)\leq\varepsilon/2.
    \end{equation*}
    This gives upper bound for $ {\rm comp}(\varepsilon,\mathcal{G}_i)$. For the proof of lower bounds consider an arbitrary algorithm $\bar X_{M,n}=\phi_{M,n}\circ\mathcal{N}_{M,n}$, such that  $\max\limits_{i=1,2}k_i=\Omega(n)$. Then the informational cost of computing of $\mathcal{N}_{M,n}$ is $\Omega(Mn)$.     If $e^{(2)}(\bar X_{M,n},\mathcal{G}_i)\leq \varepsilon$ then by Theorem \ref{lower_b_g12} we get
    \begin{equation*}
        \hat Cn^{-\gamma}\leq \varepsilon, \ \hat C\delta(M)\leq \varepsilon,
    \end{equation*}
    and hence
    \begin{eqnarray*}
        n\geq (\hat C/\varepsilon)^{1/\gamma}, \ M\geq\delta^{-1}(\varepsilon/\hat C).
    \end{eqnarray*}
    This implies the lower bound for $ {\rm comp}(\varepsilon,\mathcal{G}_i)$.
\end{proof}
\begin{remark} For example, if $\gamma=1/2$ and $\delta(M)=\Theta(M^{-\alpha+1/2})$, $\alpha\in [1,+\infty)$, then, by Theorem \ref{thm_complex}, the complexity is $\Theta\bigg((1/\varepsilon)^{\frac{4\alpha}{2\alpha - 1}}\bigg)$. Hence, if $\alpha=1$ then the minimal cost is $\Theta(\varepsilon^{-4})$, while in the case of finite dimensional $W$ the minimal cost is equal to $\Theta(\varepsilon^{-2})$ with $M$ embedded in the error constant. This example shows that the complexity significantly increases when we switch from finite to infinite dimensional driving Wiener process. 
\end{remark}
\begin{remark}
    \label{class_c1}
  It turns out that
    \begin{eqnarray*}
    &&\mathcal{C}_1(p,D,L,\varrho_2,\nu)=\{c:[0,T]\times\mathbb{R}^d\times\mathbb{R}^{d'}\mapsto\mathbb{R}^d \ | \ \|c(0,0,0)\|\leq D/\lambda^{1/p}, \notag\\
    &&\quad\quad \|c(t_1,x_1,y_1)-c(t_2,x_2,y_2)\|\leq (L/\lambda^{1/p})(\|x_1-x_2\|+(1+\|x_2\|)|t_1-t_2|^{\varrho_2})\notag\\
    &&\quad\quad\quad \hbox{for all} \ (t_1,x_1,y_1),(t_2,x_2,y_2)\in [0,T]\times\mathbb{R}^d\times\mathbb{R}^{d'}
    \}.
    \end{eqnarray*}
\end{remark}
\begin{remark}
    \label{sub_c_2} If the measure $\nu$ satisfies $(D)$ then $\mathcal{\tilde C}_2(p,D,L,\varrho_2,\nu)\subset\mathcal{C}_2(p,D,L,\varrho_2,\nu)$, where
    \begin{eqnarray*}
    &&\mathcal{\tilde C}_2(p,D,L,\varrho_2, \nu)=\{c:[0,T]\times\mathbb{R}^d\times\mathbb{R}^{d'}\mapsto\mathbb{R}^{d} \ | \ \exists_{\tilde c\in\mathcal{\tilde C}(p,D/\kappa_p,L/\kappa_p,\varrho_2)}:c(t,x,y)=\tilde c(t,x)y\notag\\ 
    &&\quad\quad\quad\quad\quad\quad\hbox{for all} \ (t,x,y)\in [0,T]\times\mathbb{R}^d\times\mathbb{R}^{d'}\},
\end{eqnarray*}
and 
\begin{eqnarray*}
&&\mathcal{\tilde C}(p,D,L,\varrho_2)=\{\tilde c:[0,T]\times\mathbb{R}^d\mapsto\mathbb{R}^{d\times d'} \ | \  \|\tilde c(0,0)\|\leq D, \|\tilde c(t,x_1)-\tilde c(t,x_2)\|\leq L\|x_1-x_2\|, \notag\\
&&\quad\quad\quad \|\tilde c(t_1,x)-\tilde c(t_2,x)\|\leq L(1+\|x\|)|t_1-t_2|^{\varrho_2} \ \hbox{for all} \ t_1,t_2,t\in [0,T], \ x_1,x_2,x\in\mathbb{R}^d \}.
\end{eqnarray*}
\end{remark}
\section{Numerical experiments and  implementation issues in CUDA C}\label{sec:experiments}
In this section we compare the obtained theoretical results with the outputs of performed simulations. Firstly, we consider the jump-diffusion Ornstein--Uhlenbeck process. Next, we focus on Black--Scholes--Merton equation with stochastic integral driven by compound Poisson process $(L(t))_{t\in[0,T]}$. In our analyses we will use the following fact.
\begin{fact}\label{fact_series_tail}
For given $\sigma \in \mathbb{R}_+,$ $\alpha \in [1,+\infty)$ let $b^{(j)}:[0,T]\times \mathbb{R}\mapsto \mathbb{R}$ be as follows $$b^{(j)}(t,x) = \frac{\sigma}{j^\alpha}x, \quad j \in \mathbb{N}.$$ Then $b=(b^{(1)}, b^{(2)}, \ldots)\in \mathcal{B}(C,D,L,\Delta,1)$ for some $C,D,L>0$ and $\delta(M) = \Theta(M^{-\alpha + 1/2})$.
\end{fact}
\noindent If $\delta(M) = \Theta(M^{-\alpha + 1/2})$, $\alpha \in [1,+\infty)$, and $\gamma=1/2$ then from Theorem \ref{thm_complex} we get for the randomized Euler algorithm $\bar X^{RE}_{M,n}$ that the optimal (up to  constants) choice of  $(M,n)$ is $\displaystyle{M(\varepsilon) = O\Bigl((1/\varepsilon)^\frac{2}{2\alpha -1}\Bigr)}$, $\displaystyle{n(\varepsilon) = O\Bigl((1/\varepsilon)^2\Bigr)}$ and hence we can take $M(\varepsilon)=O\Bigl((n(\varepsilon))^{\frac{1}{2\alpha-1}}\Bigr)$. The error of $\bar X^{RE}_{M,n}$, expressed in terms of the informational cost, is $O\Bigl((cost(\bar X^{RE}_{M,n}))^{\frac{1}{4\alpha}-\frac{1}{2}}\Bigr)=O(\varepsilon)$. Hence, the slope of  regression lines computed for the $\log(error)$ vs $\log(cost)$ scale should be close to $\frac{1}{4\alpha}-\frac{1}{2}$.
\subsection{Ornstein--Uhlenbeck process with jumps}

Now we consider the following equation
\begin{equation}
\label{sim:ornuhl}
    X(t) = \eta + \int\limits_0^t \left(\mu - A X(s)\right)\text{d}s + \sum\limits_{j = 1}^{+\infty}\int\limits_{0}^t \frac{\sigma_j}{j^\alpha}\text{d}W_j (s) + \int\limits_0^t c_1(s)\text{d}N(s), \quad t\in[0,T],
\end{equation}
where $A,\mu \in \mathbb{R},$ $\alpha >1,$ and $(\sigma_j)_{j=1}^{+\infty}$ is a bounded sequence of positive real numbers, $c_1$ is a given function, and $N=(N(t))_{t\in[0,T]}$ is a Poisson process with intensity $\lambda >0.$ The solution of the equation \eqref{sim:ornuhl} is of the form of
\begin{eqnarray}\label{sim:ornuhl_solution}
        X(t) = e^{-At}\biggr(\eta + \mu \int\limits_{0}^{t}e^{As}\text{d}s + \sum_{j=1}^{+\infty}\frac{\sigma_j}{j^\alpha}\int\limits_{0}^t e^{As}\text{d}W_j(s) + \int\limits_{0}^{t}e^{As}c_1(s)\text{d}N(s)\biggr).
\end{eqnarray}
For the simulation purposes, we set $\sigma_j = \sigma = 0.4, j \in \mathbb{N}, \ $ $c_1(t)=t, \ $ $T=1.53, \ \mu = 0.08, \ \alpha = 1.2, 
\ \lambda = 1.21.$  Note that while the analytical formula \eqref{sim:ornuhl_solution} is known, it involves the stochastic integrals calculation. 
Therefore, we simultaneously execute schemes $\bar{X}_{M,n}^{RE}$ and $\bar{X}_{10M, 100n}^{RE}$ based on common rare grid and fine grid, respectively.
We estimate the $L^2(\Omega)$--error in the following way
\begin{eqnarray}\label{err1}
    \hat{e}^{(2)}_{K}(\bar{X}_{M,n}^{RE}, (a,b,c,\eta)) = \left(\frac{1}{K}\sum_{l=1}^{K}\big|\bar{X}_{10M,100n,l}^{RE}(a,b,c,\eta) - \bar{X}_{M,n,l}^{RE}(a,b,c,\eta)\big|^2\right)^{1/2}
\end{eqnarray}
with $K=10^5$ trajectories and $n=\lfloor 10\cdot M^{1.4} \rfloor .$ 
By the virtue of Corollary \ref{upper_bound_12} and Fact \ref{fact_series_tail}
\begin{eqnarray*}
        \hat{e}^{(2)}_K(\bar{X}_{M,n}^{RE}, (a,b,c,\eta))  = O ((cost(\bar X^{RE}_{M,n}))^{-7/24}).
\end{eqnarray*}
We take $M = \lfloor{20 \cdot 1.3^{i/4}\rfloor}, \ i=0,1,\ldots, 19.$ In Figure \ref{fig:ornuhl} we can see that the obtained slope coefficient $(-0.288)$ almost perfectly matches the predicted one $(-0.292)$.
\begin{figure*}[h!]
    \centering
    \begin{subfigure}[t]{0.5\textwidth}
        \centering
        \includegraphics[width=1.0\textwidth]{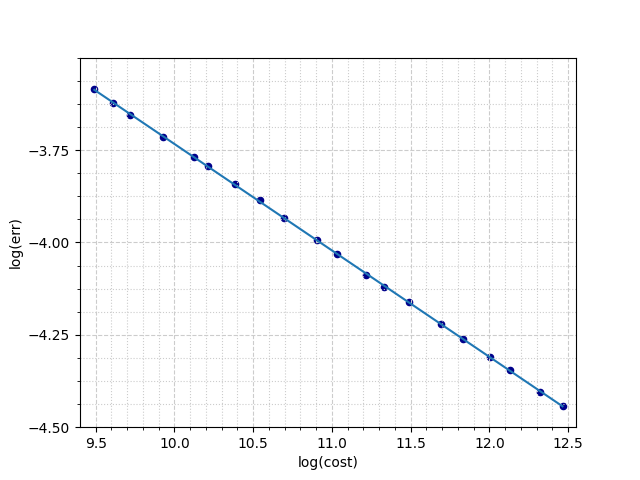}
        \caption{Ornstein--Uhlenbeck process with jumps}
        \label{fig:ornuhl}
    \end{subfigure}%
    ~ 
    \begin{subfigure}[t]{0.5\textwidth}
        \centering
        \includegraphics[width=1.0\textwidth]{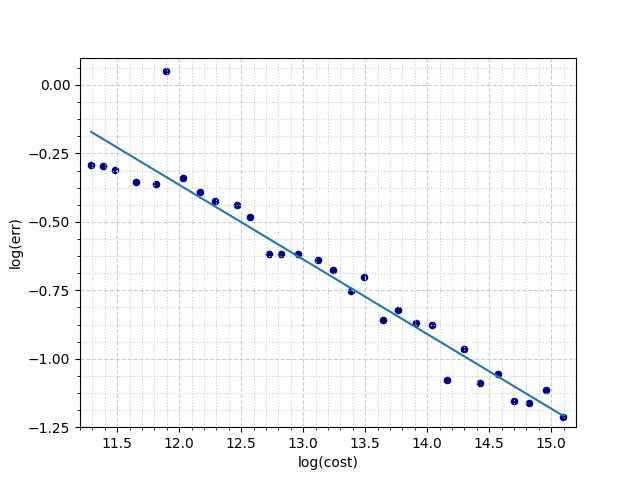}
        \caption{Merton model with compound Poisson process}
        \label{fig:merton}
    \end{subfigure}
    \caption{The  $\log(error)$ vs $\log(cost)$ plots.}
\end{figure*}
\FloatBarrier

\subsection{Merton model}
Let us consider the following equation
\begin{equation}
\label{sim:infBSMC}
    X(t) = \eta + \int\limits_0^t \mu X(s)\text{d}s + \sum_{j=1}^{+\infty}\int\limits_{0}^t \frac{\sigma_j}{j^\alpha}X(s)\text{d}W_j (s) + \int\limits_0^t X(s-)\text{d}L(s), \quad t\in [0,T],
\end{equation}
where $\mu \in \mathbb{R},$ $\alpha \geq 1,$ $(\sigma_j)_{j=1}^{+\infty}$ is a bounded sequence of positive real numbers, and $L=(L(t))_{t\in[0,T]}$ is a compound Poisson process with intensity $\lambda >0$ and jump heights $(\xi_{i})_{i=1}^{+\infty}.$ 
The solution of the equation \eqref{sim:infBSMC} can be described by the following formula
\begin{eqnarray*}
        X(t) = \eta\exp\biggr[\biggr(\mu -\frac{1}{2} \sum_{j=1}^{+\infty}\frac{\sigma_j^2}{j^{2\alpha}}\biggr)t + \sum_{j=1}^{+\infty}\frac{\sigma_j}{j^\alpha}W_j(t)\biggr]\prod_{i=1}^{N(t)}(1+\xi_{i}).
\end{eqnarray*}
For the simulation purposes, we set $\sigma_j = \sigma = 0.4, j \in \mathbb{N},$ $T=1.53, \mu = 0.08, \alpha = 1, \eta=1, \lambda = 1.21.$ Let $(Y_{i})_{i=1}^{+\infty}$ be a sequence of  independent random variables that are normally distributed with zero mean and unit variance. We assume that the jump heights sequence of random variables is defined by
$\xi_{i} = -0.5\one_{(-\infty, 0]}(Y_{i}) + (0.5+Y_{i})\one_{(0, +\infty)}(Y_{i}).$ We estimate the error in the $L^2(\Omega)$ norm in the following way
\begin{equation}\label{err2}
    \hat{e}^{(2)}_{K}(\bar{X}_{M,n}^{RE}, (a,b,c,\eta)) = \left(\frac{1}{K}\sum_{l=1}^{K}\big|X_{l}^{10M}(T) - \bar{X}_{M,n,l}^{RE}(a,b,c,\eta)\big|^2\right)^{1/2},
\end{equation}
where $K=5 \cdot 10^5$ is the number of trajectories and $X_{l}^{10M}(T)$ is the $l$th sample of $X^{10M}(T).$ We also set $n=200M,$ where $M=\lfloor 20 \cdot 1.3^{0.25 i} \rfloor, \ i=0,1,\ldots, 29.$ By Corollary \ref{upper_bound_12} and Fact~\ref{fact_series_tail}
\begin{eqnarray*}
       e^{(2)}(\bar{X}_{M,n}^{RE},(a,b,c,\eta)) = O(cost(\bar X^{RE}_{M,n})^{-1/4}).
\end{eqnarray*}
The method in Figure \ref{fig:merton} exhibits convergence rate equal to $-0.272.$ 
\subsection{Details of implementation in CUDA C}
In this section we present the crucial parts of the code of $\bar{X}_{M,n}^{RE}$ algorithm implemented in CUDA C and executed on NVIDIA Titan V GPU (Graphics Processing Unit). The details of CUDA C concept can be found, for instance, in \cite{AKAPhD}. In addition in \cite{AKAPhD} the author presented implementation of Milstein scheme (based on equidistant and nonequidistant mesh) and used it for optimal approximation of solutions of SDEs driven by Wiener and Poisson processes. The architecture of GPU enables to significantly decrease computation time by simulating multiple trajectories in parallel for e.g. Monte Carlo approximation; see \cite{AKPMPP} for performance comparison between CPU (Computer Processing Unit) and GPU. In particular, this refers to the errors estimation as per \eqref{err1} and \eqref{err2}. 

The current implementation solution consists of several separate $\texttt{.cu}$ and $\texttt{.cuh}$ files in order to maintain the code brevity. The input data (including model choice, solution formula if available, input parameters' values) is sourced from $\texttt{input.cu}$ file. In main function located within $\texttt{kernel.cu}$ the user defines whether exact solution is known or rare/fine grid approach should be leveraged, as well as whether the underlying Wiener process is countably dimensional. Note that the user may also specify the relation between parameters $M,n.$ Consequently, $\texttt{investigate\_error\_w\_exact}$ or $\texttt{investigate\_error\_w\_unknown}$ function is executed. Moreover, the jumps sequence (if applicable) for a single trajectory is generated within $\texttt{jumps.cu}$ file.

We provide three listings which illustrate the critical parts of our code together with relevant comments. The comments linked to one particular line are appended at the end of this line while the comments referring to a certain part of the code are appended before the corresponding code fragment.
\begin{lstlisting}[language=C, caption=Fragment of the code where the trajectories are split between separate available threads. , basicstyle=\ttfamily\tiny]
__host__ double investigate_error_w_unknown(int rare_grid_density, int rare_wiener_dim, int fine_grid_ratio, int fine_wiener_dim, int trajectories, int power) {
    // (1) 
    int iterations = trajectories >= MAX_TR_BLOCKS ? trajectories / MAX_TR_BLOCKS : 1;
    trajectories = MAX_TR_BLOCKS;
    
    // (2)
    srand(time(NULL));
    double* errors = (double*)malloc(sizeof(double) * trajectories);
    double* errors_dev;

    // (3)
    cudaEvent_t start, stop;
    HANDLE_ERROR(cudaEventCreate(&start));
    HANDLE_ERROR(cudaEventCreate(&stop));
    HANDLE_ERROR(cudaEventRecord(start, 0));
    
    HANDLE_ERROR(cudaMalloc((void**)&errors_dev, sizeof(double) * trajectories)); // (4)

    double result = 0.0;
    for (int i = 0; i < iterations; i++) {
        calculate_scheme_error << <min(trajectories, MAX_BLOCKS), 1 >> > (errors_dev, fine_wiener_dim, fine_grid_ratio, rare_wiener_dim, rare_grid_density, trajectories, power, rand()); // (5)
        
        HANDLE_ERROR(cudaMemcpy(errors, errors_dev, sizeof(double) * trajectories, cudaMemcpyDeviceToHost)); // (6) 

        for (int j = 0; i < trajectories; i++) { // (7)
            result += errors[j] / (trajectories*iterations);
        }
        printf(" %d/%d Done\n", i + 1, iterations);
    }
    // (8)
    HANDLE_ERROR(cudaFree(errors_dev));
    HANDLE_ERROR(cudaEventRecord(stop, 0));
    cudaEventSynchronize(stop);
    float time;
    HANDLE_ERROR(cudaEventElapsedTime(&time, start, stop));
    cudaEventDestroy(start);
    cudaEventDestroy(stop);
    free(errors);
    
    printf("Error investigation with unknown: %3.1f ms for {dim_w:%d, dim_W:%d, n:%d, N:%d} (%d samples).\n", time, rare_wiener_dim, fine_wiener_dim, rare_grid_density, rare_grid_density*fine_grid_ratio, trajectories); // (9)
    
    return pow(result, 1.0/power); // (10)
} 
\end{lstlisting}
\begin{enumerate}
\item[(1)] Dividing given number of trajectories into MAX\_TR\_BLOCKS number of blocks that run in parallel for the specified number of iterations. 
\item[(2)] Initializing pseudo-random number generator (PRN).
\item [(3)] Starting performance measurement by creating so called time events which will return elapsed time of GPU computation.
\item [(4)] Allocating memory for the kernel function output (partial results computed in parallel).
\item [(5)] Choosing at most 65535 of the available blocks, which is the number determined by the GPU architecture.
\item [(6)] Copying the results from device to host and finishing performance measurement.
\item [(7)] Calculating the average error for all trajectories.
\item [(8)] Releasing memory allocated for the kernel function output and the results copied from the GPU.
\item [(9)] Printing the current set of parameters together with execution time for the user convenience.
\item [(10)] Returning the value of Monte Carlo estimator for the scheme error in the $p$-th norm.
\end{enumerate}
\begin{lstlisting}[language=C, caption=Fragment of the code where the scheme error is investigated using parallel processing. , basicstyle=\ttfamily\tiny]
__global__ void calculate_scheme_error(double* errors, int fine_wiener_dim, int fine_grid_ratio, int rare_wiener_dim, int rare_grid_density, int trajectories_num, double power, long seed) {
    int trajectory_index = blockIdx.x * blockDim.x + threadIdx.x; // (1)
    // (2) 
    curandState_t state;
    curand_init(seed + threadIdx.x + blockDim.x * blockIdx.x, 0, 0, &state);

    while (trajectory_index < trajectories_num) {
        double fine_X = x0;
        double rare_X = x0;
        // (3) 
        double* fine_wiener_increment = (double*)malloc(sizeof(double) * fine_wiener_dim);
        double* rare_wiener_increment = (double*)malloc(sizeof(double) * rare_wiener_dim);
        Jump* jumps_head = (Jump*)malloc(sizeof(Jump));
        // (4)
        generate_jumps(&state, INTENSITY, T, jumps_head);
        Jump* rare_grid_jump = jumps_head;
        Jump* fine_grid_jump = jumps_head;
        // (5)
        double H = T / (rare_grid_density * fine_grid_ratio);
        double h = T / rare_grid_density;

        double ti, tj;
        for (int i = 0; i < rare_grid_density; i++) { // (6)
            ti = i * h;
            for (int k = 0; k < rare_wiener_dim; k++) {
               rare_wiener_increment[k] = 0.0; // (7)
            }
            rare_grid_jump = first_jump_after_time(rare_grid_jump, ti); // (8)
            fine_grid_jump = rare_grid_jump;
            for (int j = 0; j < fine_grid_ratio; j++) {
                tj = ti + j * H;
                for (int k = 0; k < fine_wiener_dim; k++) {
                fine_wiener_increment[k] = curand_normal(&state) * sqrt(H); // (9)
                    if (k < rare_wiener_dim) {
                        rare_wiener_increment[k] += fine_wiener_increment[k];
                    }
                }
                fine_grid_jump = first_jump_after_time(fine_grid_jump, tj);
                fine_X = random_euler_single_step(fine_X, fine_wiener_increment, fine_wiener_dim, fine_grid_jump, tj, tj + H, state); // (10)           
            }
            rare_X = random_euler_single_step(rare_X, rare_wiener_increment, rare_wiener_dim, rare_grid_jump, ti, ti + h, state);
        }
        errors[trajectory_index] = pow(abs(fine_X - rare_X), power); // (11)
        // (12)
        free_jumps_list(jumps_head);
        free(fine_wiener_increment);
        free(rare_wiener_increment);
        trajectory_index += gridDim.x * blockDim.x; // (13)
    }
}
\end{lstlisting}

\begin{enumerate}
    \item [(1)] Assigning every single trajectory to the separate kernel function block in order to be run in parallel.
    \item [(2)] Obtaining distinct random numbers per simulation and trajectory by initializing our PRN generator with particular seed (current time) and trajectory index.
    \item [(3)] Allocating memory for the truncated Wiener process increments generated on rare and fine grid, and for the head of jumps list.
    \item [(4)] Independently generating sequence of jump times for the current trajectory.
    \item [(5)] Initialising step size per grid.
    \item [(6)] Starting the outer loop (indexed with 'i') in order to iterate through rare grid points. The inner loop (indexed with 'j') iterates through fine grid points.
    \item [(7)] Assigning every rare Wiener increment coordinate value equal to zero.
    \item [(8)] Retrieving subsequent jump time given the previously located jump.
    \item [(9)] Assigning the value of truncated Wiener increment on the fine grid. For this purpose, we use CUDA C sampling from normal distribution.
    \item [(10)] Calculating the randomized Euler scheme's single step in order to find the approximated value of scheme in the current time point.
    \item [(11)] Calculating and saving the $p$-th power of the absolute difference between approximated solution on rare and dense grid.
    \item [(12)] Releasing previously allocated memory.
    \item [(13)] If the number of trajectories is greater than the number of blocks, we assign another trajectory to the currently running kernel function block and start another simulation.
\end{enumerate}

\begin{lstlisting}[language=C, caption=Fragment of the code where the single step of truncated dimension randomized Euler algorithm is executed. , basicstyle=\ttfamily\tiny]
__device__ double random_euler_single_step(double prev_X, double* wiener_increment, int wiener_dim, Jump* jumps, double time_from, double time_to, curandState_t state) {
    double X = prev_X;
    
    X += func_a(time_from + curand_uniform_double(&state) * (time_to - time_from), prev_X) * (time_to - time_from); // (1)
    
    for (int k = 0; k < wiener_dim; k++) { // (2)
        X += func_b(k + 1, time_from, prev_X) * wiener_increment[k];
    }
    
    if (jumps != NULL) { // (3)
        while (jumps->time < time_to) {
            if (jumps->time > time_from) {
                X += func_c(jumps->time, prev_X, jumps->height);
            }
            jumps = jumps->next_jump;
            if (jumps == NULL) break;
        }
    }
    return X;
}
\end{lstlisting}

\begin{enumerate}
    \item Adding drift-related term.
    \item Adding diffusion-related term.
    \item Adding jump part-related term.
\end{enumerate}

\section{Conclusions}
We investigated complexity bounds in certain input data classes for the pointwise approximation of the systems of SDEs which contain integrals with respect to countably dimensional Wiener process and random Poisson measure. We presented the implementable truncated dimension randomized Euler scheme $\bar X_{M,n}^{RE}$, derived its convergence rate $O(n^{-\min\{\varrho_1, \varrho_2, 1/p \}} + \delta(M))$ together with optimality in the class of input data which are important from the point of view of possible applications. Our theoretical results are supported by numerical experiments performed in CUDA architecture embedded into high level programming language C. The usage of GPU in this case is justified by high dimensionality and complexity of intermediate computations.
Finally, we conjecture that the lower error bound in the considered setting  also depends on $p$, i.e., $1/p$ should  also be present in the exponent of the error bound. 
\section{Appendix}
The proof of following fact is straightforward, so we skip the details.
\begin{fact} 
\label{fact_1}
\begin{itemize}
    \item [(i)]  If $a\in \mathcal{A}(D,L)$ then for all $(t,x)\in [0,T]\times\mathbb{R}^d$
     \begin{equation*}
        \|a(t,x)\|\leq \max\{D,L\}(1+\|x\|).
    \end{equation*}
    \item [(ii)] If $b\in \mathcal{B}(C,D,L,\Delta,\varrho_1)$ then for all $(t,x)\in [0,T]\times\mathbb{R}^d$
    \begin{equation*}
        \|b(t,x)\|\leq \max\{D+LT^{\varrho_1},L\}(1+\|x\|).
    \end{equation*}
    \item[(iii)] Let $b\in \mathcal{B}(C,D,L,\Delta,\varrho_1)$. Then for all $M\in\mathbb{N}$ we have that $b^M\in \mathcal{B}(C,D,L,\Delta,\varrho_1)$ and for all $(t,x)\in [0,T]\times\mathbb{R}^d$
    \begin{equation*}
        \|b^M(t,x)\|\leq \max\{D+LT^{\varrho_1},L\}(1+\|x\|).
    \end{equation*}
    \item[(iv)] Let $b\in \mathcal{B}_0(C,D,L,\Delta,\varrho_1)$. Then $b^M\in \mathcal{B}_0(C,D,L,\Delta,\varrho_1)$ for any $M\in\mathbb{N}$.
    \item [(v)] If $c\in\mathcal{C}(p,D,L,\varrho_2, \nu) $ then for all $(t,x) \in [0,T]\times \mathbb{R}^d$
    \begin{equation*}
	\biggr(\int\limits_{\mathcal{E}} \|c(t,x,y)\|^p \  \nu(\rd y) \biggr)^{1/p} \leq \max\{L,LT^{\varrho_2}+D\} (1+ \|x\|).
    \end{equation*}
 \item [(vi)] If $(a,c)\in \mathcal{A}(D,L)\times \mathcal{C}(p,D,L,\varrho_2, \nu)$ then the mapping $\tilde a$, defined as in \eqref{ta_def}, is Borel measurable and for all $t\in [0,T]$, $x,y\in\mathbb{R}^d$ 
  \begin{equation*}
     \|\tilde a(t,x)\|\leq\Bigl(\max\{D,L\}+\lambda^{\frac{p}{p-1}}\max\{L,LT^{\varrho_2}+D\}\Bigr)(1+\|x\|),
 \end{equation*}
 \begin{equation*}
     \|\tilde a(t,x)-\tilde a(t,y)\|\leq (1+\lambda^{\frac{p}{p-1}})L\|x-y\|.
 \end{equation*}
\end{itemize}
\end{fact}
{\noindent\bf Proof of Lemma \ref{lemma_sol}. } The proof of \eqref{lemma_sol_estimate} follows the usual localization argument, see, for example, \cite{situ}, \cite{sabanis}. The main part of the proof consists of use of the Burkholder and Kunita inequalities and the fact that the constant in the linear growth bound for $b^M$ does not depend on $M$.
Since we were not able to find a direct reference in literature that covers the case considered in this paper, for a convenience of the reader we provide a complete argumentation.

Let us fix $M\in \mathbb{N}\cup\{\infty\}$ and define the stopping time $\tau_R=\inf\{t\geq 0 \ | \ \|X^M(t)\|>R\}\wedge T$, $R\in\mathbb{N}$. (Recall that we take $X^{\infty}=X$). Since trajectories of $X^M$ are c\`adl\`ag, we get
\begin{eqnarray}
\label{limit_tau_r}
        \mathbb{P}\Bigl(\bigcup_{R\geq 1}\{\tau_R=T\}\Bigr)=1.
\end{eqnarray} 
 Furthermore, $\|X^M(t-)\|\leq R$ for all $0\leq t\leq \tau_R$. This and Fact \ref{fact_1} imply that for $f\in\{\tilde a,b^M\}$, $t\in [0,T]$
 \begin{equation*}
    \mathbb{E} \int\limits_{0}^{t\wedge\tau_R}\|f(s,X^M(s))\|^p \rd s= \mathbb{E} \int\limits_{0}^{t\wedge\tau_R-}\|f(s,X^M(s))\|^p \rd s\leq K(1+R)^p<+\infty,
 \end{equation*}
 and
 \begin{equation*}
     \mathbb{E}\int\limits_{0}^{t\wedge\tau_R}\int\limits_{\mathcal{E}}\|c(s,X^M(s-),y)\|^p\nu(\rd y)\rd s\leq K(1+R)^p<+\infty.
 \end{equation*}
Hence, the processes $\displaystyle{\Biggl(\int\limits_{0}^{t\wedge\tau_R} b^M(s,X^M(s))\text{d}W(s),\Sigma_t\Biggr)_{t\in [0,T]}}$ and $\displaystyle{\Biggl(\int\limits_{0}^{t\wedge\tau_R}\int\limits_{\mathcal{E}} c(s,X^M(s-),y)\tilde N(\rd y,\rd s),\Sigma_t\Biggr)_{t\in [0,T]}}$ are $L^p(\Omega)$-martingales. By using the Burkholder and Kunita inequalities we obtain for all $t\in [0,T]$, $R\in\mathbb{N}$
\begin{eqnarray}
        &&\mathbb{E} \|X^M(t\wedge\tau_R)\|^p\leq K\mathbb{E}\|\eta\|^p+ KT^{p-1}\mathbb{E}\int\limits_{0}^{t\wedge \tau_R}\|\tilde a(s,X^M(s))\|^p \rd s\notag\\
        && + K\mathbb{E}\Biggl(\int\limits_{0}^{t\wedge \tau_R}\|b^M(s,X^M(s))\|^2 \rd s\Biggr)^{p/2}+K\mathbb{E}\Biggl(\int\limits_{0}^{t\wedge \tau_R}\int\limits_{\mathcal{E}}\|c(s,X^M(s-),y)\|^2\nu(\rd y)\rd s\Biggr)^{p/2}\notag\\
        &&+K\mathbb{E}\int\limits_{0}^{t\wedge \tau_R}\int\limits_{\mathcal{E}}\|c(s,X^M(s-),y)\|^p\nu(\rd y)\rd s\leq K_1(1+\mathbb{E}\|\eta\|^p)+K_2\mathbb{E}\int\limits_{0}^{t\wedge\tau_R}\|X^N(s)\|^p \rd s,\notag
\end{eqnarray}
which, in particular, implies that
\begin{equation*}
    \sup\limits_{0\leq t\leq T}\mathbb{E} \|X^M(t\wedge\tau_R)\|^p\leq K_1(1+\mathbb{E}\|\eta\|^p)+K_2TR^p<+\infty, 
\end{equation*}
and
\begin{equation*}
    \mathbb{E} \|X^M(t\wedge\tau_R)\|^p\leq K_1(1+\mathbb{E}\|\eta\|^p)+K_2\int\limits_{0}^{t}\mathbb{E}\|X^M(s\wedge \tau_R)\|^p \rd s.
\end{equation*}
Since the function $[0,T]\ni t\mapsto \mathbb{E} \|X^M(t\wedge\tau_R)\|^p$ is bounded and Borel, by the Gronwall's lemma we get the following (independent of $R$) bound $\mathbb{E} \|X^M(t\wedge\tau_R)\|^p\leq K_3 (1+\mathbb{E}\|\eta\|^p)$ for all $t\in [0,T]$. By applying Fatou's lemma and \eqref{limit_tau_r} we get
\begin{equation}
\label{moment_sol_est_1}
    \sup\limits_{0\leq t\leq T}\mathbb{E}\|X^M(t)\|^p\leq K_3 (1+\mathbb{E}\|\eta\|^p).
\end{equation}
By using \eqref{moment_sol_est_1} together with the H\"older, Burkholder and Kunita inequalities we obtain \eqref{lemma_sol_estimate}. Namely,
\begin{eqnarray*}
        &&\mathbb{E}\Bigl(\sup\limits_{0\leq t \leq T} \|X^M(t)\|^p\Bigr)\leq K\mathbb{E}\|\eta\|^p+ KT^{p-1}\mathbb{E}\int\limits_{0}^{T}\|\tilde a(s,X^M(s))ds\|^p \rd s\notag\\
        && + K\mathbb{E}\Biggl(\int\limits_{0}^{T}\|b^M(s,X^M(s))\|^2 \rd s\Biggr)^{p/2}+K\mathbb{E}\Biggl(\int\limits_{0}^{T}\int\limits_{\mathcal{E}}\|c(s,X^M(s-),y)\|^2\nu(\rd y)\rd s\Biggr)^{p/2}\notag\\
        &&+K\mathbb{E}\int\limits_{0}^{T}\int\limits_{\mathcal{E}}\|c(s,X^M(s-),y)\|^p\nu(\rd y)\rd s\leq K_1(1+\mathbb{E}\|\eta\|^p)+K_2\int\limits_{0}^{T}\mathbb{E}\|X^M(s)\|^p \rd s\leq C_1.\notag
\end{eqnarray*}
We now justify (i)-(iii). By applying the H\"older, Burkholder and Kunita inequalites,  Fact \ref{fact_1}, and \eqref{lemma_sol_estimate}, we obtain  for all $N\in\mathbb{N}\cup\{\infty\}$, $s,t\in [0,T]$ that
\begin{displaymath}
\mathbb{E} \|X^M(t)-X^M(s)\|^p=\mathbb{E}\|X^M(t\vee s)-X^M(t\wedge s)\|^p
\end{displaymath}
\begin{eqnarray*}
        &&\leq  K(t\vee s-t\wedge s)^{p-1}\mathbb{E}\int\limits_{t\wedge s}^{t\vee s}\|\tilde a(s,X^M(s))\|^p \rd s + K(t\vee s-t\wedge s)^{\frac{p}{2}-1}\mathbb{E}\int\limits_{t\wedge s}^{t\vee s}\|b^M(s,X^M(s))\|^p \rd s\notag\\
        &&+K\Bigl((\lambda (t\vee s-t\wedge s))^{\frac{p}{2}-1}+1)\mathbb{E}\int\limits_{t\wedge s}^{t\vee s}\int\limits_{\mathcal{E}}\|c(s,X^M(s-),y)\|^p\nu(\rd y)\rd s
\end{eqnarray*}
\begin{eqnarray*}
        &&\leq K_1 \Bigl(1+\mathbb{E}(\sup\limits_{0\leq t\leq T}\|X^M(t)\|^p)\Bigr)|t-s|^p+K_1 \Bigl(1+\mathbb{E}(\sup\limits_{0\leq t\leq T}\|X^M(t)\|^p)\Bigr)|t-s|^{\frac{p}{2}}\notag\\
        &&+K_1((\lambda T)^{\frac{p}{2}-1}+1)\Bigl(1+\mathbb{E}(\sup\limits_{0\leq t\leq T}\|X^M(t)\|^p)\Bigr)|t-s|,
\end{eqnarray*}
from which  (i), (ii), and (iii) follow. \ \ \ $\square$
\newline\newline
{\bf Proof of Proposition \ref{aux_lem_1}.}
 For any $t\in [0,T]$, $M\in\mathbb{N}$, $(a,b,c,\eta)\in\calf(p,C,D,L,\Delta,\varrho_1, \varrho_2, \nu)$ we have that
	\begin{equation*}
		\mathbb{E}\|X^M(t)-X(t)\|^p \leq 4^{p-1}\left(\mathbb{E}\|A^M(t)\|^p + \mathbb{E}\|B^M_1(t)\|^p+\mathbb{E}\|B^M_2(t)\|^p + \mathbb{E}\|C^M(t)\|^p \right),
	\end{equation*}
	where
	\begin{equation*}
		\mathbb{E}\|A^M(t)\|^p = \mathbb{E}\biggr\|\int\limits_{0}^{t}\left(a(s,X(s))-a(s,X^M(s))\right)\rd s\biggr\|^p,
	\end{equation*}
	\begin{equation*}
		\mathbb{E}\|B^M_1(t)\|^p = \mathbb{E}\biggr\|\int\limits_{0}^{t}\Bigl(b(s,X(s))-b(s,X^M(s))\Bigr)\rd W(s)\biggr\|^p,
	\end{equation*}
	\begin{equation*}
		\mathbb{E}\|B^M_2(t)\|^p = \mathbb{E}\biggr\|\int\limits_{0}^{t}(b-P_M b)(s,X^M(s))\rd W(s)\biggr\|^p,
	\end{equation*}
	\begin{equation*}
	\mathbb{E}\|C^M(t)\|^p = \mathbb{E}\biggr\|\int\limits_{0}^{t}\int\limits_{\mathcal{E}}\left(c(s,X(s-),y)-c(s,X^M(s-),y)\right)N(\rd y,\rd s)\biggr\|^p.
	\end{equation*}
	Firstly, by the H\"{o}lder inequality
	\begin{equation}\label{lemma_A_N}
		\mathbb{E}\|A^M(t)\|^p 
		\leq T^{p-1}L^p\mathbb{E}\int\limits_{0}^{t}\|X^M(s) - X(s)\|^p \rd s.
	\end{equation}
	From (B3), (B4), Lemma \ref{lemma_sol} and by the Burkholder and H\"{o}lder inequalities, we get for $t \in [0,T]$ that
	\begin{equation}
	\label{lemma_B_N_1}
	        \mathbb{E}\|B^M_1(t)\|^p\leq K_1\int\limits_0^t\mathbb{E}\|X(s)-X^M(s)\|^p \rd s,
	\end{equation}
	\begin{eqnarray}
	\label{lemma_B_N_2}
	        &&\mathbb{E}\|B^M_2(t)\|^p \leq K_2\int\limits_0^t\mathbb{E}\|(b-P_Mb)(s,X^M(s))\|^p \rd s\notag\\
	        &&\leq K_3\Bigl(1+\sup\limits_{0\leq s\leq T}\mathbb{E}\|X^M(s)\|^p\Bigr) (\delta(M))^p\leq K_4(\delta(M))^p.
	\end{eqnarray}
Finally, from (C3) and by using the Kunita and H\"older inequalities we obtain
	\begin{eqnarray}
	\label{lemma_C_N}
	   &&\mathbb{E}\|C^M(t)\|^p \leq 2^{p-1} \mathbb{E}\biggr\|\int\limits_{0}^{t}\int\limits_{\mathcal{E}}\left(c(s,X(s-),y)-c(s,X^M(s-),y)\right)\tilde{N}(\rd y,\rd s)\biggr\|^p\notag \\
	   &&\quad\quad\quad\quad\quad + 2^{p-1}\mathbb{E}\biggr\|\int\limits_{0}^{t}\int\limits_{\mathcal{E}}\left(c(s,X(s-),y)-c(s,X^M(s-),y)\right)\nu(\rd y)\rd s\biggr\|^p\notag\\
	   &&\leq K_5\mathbb{E}\int\limits_0^t\Biggl(\int\limits_{\mathcal{E}}\|c(s,X(s-),y)-c(s,X^M(s-),y)\|^p\nu(\rd y)\Biggr)\rd s\notag\\
	   &&\leq K_5L^p\mathbb{E}\int\limits_0^t\|X(s-)-X^M(s-)\|^p \rd s=K_5L^p\mathbb{E}\int\limits_0^t\|X(s)-X^M(s)\|^p \rd s.
	\end{eqnarray}
	Combining \eqref{lemma_A_N}, \eqref{lemma_B_N_1}, \eqref{lemma_B_N_2} and \eqref{lemma_C_N} we have for $t \in [0,T]$ that
	\begin{equation*}
		\mathbb{E}\|X(t)-X^M(t)\|^p \leq K_4 (\delta(M))^p + K_6 \int\limits_{0}^{t}\mathbb{E}\|X(s)-X^M(s)\|^p \rd s.
	\end{equation*}
	By Lemma \ref{lemma_sol} and Tonelli's theorem  the function $[0,T]\ni t \mapsto  \mathbb{E}\|X^M(t) - X(t)\|^p$ is Borel measurable and bounded. Hence, application of the Gronwall's lemma yields for all $t\in [0,T]$
	\begin{equation*}
		\mathbb{E}\| X(t)-X^M(t)\|^p \leq K_7 (\delta(M))^p,
	\end{equation*}
	with $K_7$ depending only on the parameters of the class $\calf(p,C,D,L,\Delta,\varrho_1, \varrho_2, \nu)$. This implies~\eqref{xxn_est_1}.
	
	Take arbitrary $M\in\mathbb{N}$, $(a,c,\eta)\in \mathcal{A}(D,L)\times\mathcal{C}(p,D,L,\varrho_2,\nu)\times\mathcal{J}(p,D)$ and any $b_1,b_2\in \mathcal{B}(C,D,L,\Delta,\varrho_1)$ such that $b_1^M(t,x)=b_2^M(t,x)$ for all $(t,x)\in [0,T]\times\mathbb{R}^d$. Then $X^M(a,b_1,c,\eta)=X(a,b_1^M,c,\eta)=X(a,b_2^M,c,\eta)=X^M(a,b_2,c,\eta)$ and by the triangle inequality we have that
	\begin{eqnarray}
	\label{low_xxn_b}
	   && \sup\limits_{(a,b,c,\eta)\in\calf(p,C,D,L,\Delta,\varrho_1, \varrho_2, \nu)}\sup\limits_{0 \leq t \leq T}\Vert X(a,b,c,\eta)(t)-X^M(a,b,c,\eta)(t)\Vert_{L^2(\Omega)}\notag\\
	   &&\geq  \frac{1}{2}  \sup\limits_{0 \leq t \leq T}\Vert X(a,b_1,c,\eta)(t)-X(a,b_2,c,\eta)(t)\Vert_{L^2(\Omega)}.
	\end{eqnarray}
	In particular, by taking $a=b_1=c=\eta=0$, $b_2=\tilde b_M$ (defined in \eqref{tilde_bN}) we get $b_1^M=b_2^M=0$,  $X(0,0,0,0)(t)=0$, $X(0,b_2,0,0)(t)=[C\delta(M)W_{M+1}(t),0,\ldots,0]^T$ which together with \eqref{low_xxn_b} implies \eqref{xxn_est_2}. Finally, \eqref{xxn_est_3} is a direct consequence of  \eqref{xxn_est_1}, \eqref{xxn_est_2}, and the H\"older inequality. \ \ \ $\square$
\begin{lemma}
\label{lemma_4}
	Let $p\in  [2,+\infty).$ There exists $K \in (0,+\infty)$, depending only on the parameters of the class $\calf(p,C,D,L,\Delta,\varrho_1, \varrho_2, \nu) $, such that for every $M,n\in\mathbb{N}$ and $(a,b,c,\eta) \in \calf(p, C,D,L,\Delta,\varrho_1, \varrho_2, \nu) $ it holds
	\begin{equation}
	\label{time_continuous_scheme_estimate}
		\sup_{0 \leq t \leq T}\mathbb{E}\|\tilde{X}_{M,n}^{RE}(t)\|^p \leq K.
	\end{equation}
\end{lemma}
\begin{proof}
First, we prove by induction that 
\begin{equation}
\label{scheme_Lp_norm}
	\max_{0\leq j \leq n-1}\mathbb{E}\Vert X_{M,n}^{RE}(t_j)\Vert^p < +\infty.
\end{equation}
Let us assume there exists $ l \in \{0,1,\ldots, n-1\}$ such that $\displaystyle{\max\limits_{0\leq j \leq l}\mathbb{E}\Vert X_{M,n}^{RE}(t_j)\Vert^p<+\infty}$ (in the case when $l=n-1$ the thesis is immediate). The set of such indices $l$ is non-empty since for $l=0$ we have $\mathbb{E}\|\eta\|^p < +\infty.$  Therefore, by the Burkholder, Kunita and H\"older  inequalities, and Fact \ref{fact_1} the following estimate holds
\begin{eqnarray*}
        &&\mathbb{E}\|X_{M,n}^{RE}(t_{l+1})\|^p\leq K_1(1+\mathbb{E}\|X_{M,n}^{RE}(t_l)\|^p)+K_2 \mathbb{E}\int\limits_{t_l}^{t_{l+1}}\|b^M(t_l, X_{M,n}^{RE}(t_l))\|^p \rd s\notag\\
        &&+ K_3\mathbb{E}\int\limits_{t_l}^{t_{l+1}}\Biggl(\int\limits_{\mathcal{E}}\|c(t_l,X_{M,n}^{RE}(t_l),y)\|^p \nu(\rd y)\Biggr)\rd s\leq K_4 (1+\mathbb{E}\|X_{M,n}^{RE}(t_l)\|^p)<+\infty.
\end{eqnarray*}
Hence $\max\limits_{0\leq j\leq l+1}\mathbb{E}\|X^{RE}_{M,n}(t_j)\|^p<+\infty$ 
and, by the rules of induction, the proof of \eqref{scheme_Lp_norm} is completed. By \eqref{main_scheme_continuous}, \eqref{schemes_coincide} and \eqref{scheme_Lp_norm}, and by using analogous argumentation as above we get for all $t\in [t_j,t_{j+1}],$ $j=0, \ldots, n-1$, that
$\mathbb{E}\|\tilde{X}_{M,n}^{RE}(t)\|^p \leq K \left(1+ \mathbb{E}\|X_{M,n}^{RE}(t_j)\|^p\right)$, and hence
\begin{equation}
\label{scheme_upper_bound1}
	\sup_{0 \leq t \leq T}\mathbb{E}\| \tilde{X}_{M,n}^{RE}(t)\|^p \leq K \left(1 + \max_{0\leq j \leq n-1}\mathbb{E}\|X_{M,n}^{RE}(t_j)\|^p \right) < +\infty.
\end{equation}
Currently  constant in the bound \eqref{scheme_upper_bound1} depends on $n$. In the second part of the proof we will show, with the help of the Gronwall's lemma, that we can obtain the bound \eqref{time_continuous_scheme_estimate} with $K$ that is independent of $n$. 

Using the same decomposition as in \eqref{lemma_X_RE_equation}, we obtain that for $t \in [0,T]$
\begin{eqnarray}
\label{continuous_scheme_approx1}
		&&\mathbb{E}\|\tilde{X}_{M,n}^{RE}(t)\|^p \leq
		K\biggr(\mathbb{E}\|\eta\|^p + \mathbb{E}\Biggl\|\int\limits_{0}^{t}\tilde{a}_{M,n} (s) \rd s\Biggl\|^p +  
		\mathbb{E}\biggr\|\int\limits_{0}^{t}\tilde{b}_n^M (s) \rd W(s) \biggr\|^p \notag \\
		&& + \mathbb{E}\biggr\| \int\limits_{0}^{t}\int\limits_{\mathcal{E}}\tilde{c}_{M,n}(y,s) N(\rd y,\rd s)\biggr\|^p \biggr).
\end{eqnarray}
By the Kunita and H\"{o}lder inequalities, and Fact \ref{fact_1} we get
\begin{eqnarray}
\label{continuous_scheme_approx_compensator}
	&&\mathbb{E} \biggr\| \int\limits_{0}^{t}\int\limits_{\mathcal{E}}\tilde{c}_{M,n}(y,s) N(\rd y,\rd s)\biggr\|^p 
		\leq K_1\mathbb{E}\int\limits_0^t\int\limits_{\mathcal{E}}\|\tilde{c}_{M,n}(y,s)\|^p\nu(\rd y)\rd s\notag\\
		&& = K_1\mathbb{E}\int\limits
		_{0}^{t}\sum\limits_{j=0}^{n-1}\Biggl(\int\limits_{\mathcal{E}} \|c(t_j, \tilde{X}_{M,n}^{RE}(t_j),y)\|^p\nu(\rd y)\Biggr) \one_{ (t_j,t_{j+1}]}(s)\rd s\notag\\
		&&\leq K_2+K_3\int\limits_0^t \sum_{j=0}^{n-1}\mathbb{E}\|\tilde{X}_{M,n}^{RE}(t_j)\|^p\one_{(t_j,t_{j+1}]}(s) \rd s.
\end{eqnarray}
In analogous way we can obtain  for all $t\in [0,T]$ that
\begin{eqnarray}
\label{est_es_dsdw}
    &&\max\Biggl\{\mathbb{E}\Biggl\|\int\limits_{0}^{t}\tilde{a}_{M,n} (s) \rd s\Biggl\|^p,\mathbb{E}\biggr\|\int\limits_{0}^{t}\tilde{b}_n^M (s) \rd W(s) \biggr\|^p\Biggr\}\notag\\
    &&\leq  K_4+K_5\int\limits_0^t \sum_{j=0}^{n-1}\mathbb{E}\|\tilde{X}_{M,n}^{RE}(t_j)\|^p\one_{(t_j,t_{j+1}]}(s) \rd s.
\end{eqnarray}
Combining \eqref{continuous_scheme_approx1}, \eqref{continuous_scheme_approx_compensator}, \eqref{est_es_dsdw}  we get for all $t\in [0,T]$
\begin{equation}
\label{csa_3}
    \sup\limits_{0\leq u\leq t} \mathbb{E}\|\tilde{X}_{M,n}^{RE}(u)\|^p \leq K_6  + K_7 \int\limits_{0}^{t}\sup_{0 \leq u \leq s}\mathbb{E}\|\tilde{X}_{M,n}^{RE}(u)\|^p \rd s,
\end{equation}
where $K_6,K_7$ depend only on the parameters of the class   $\calf(p,C,D,L,\Delta,\varrho_1, \varrho_2, \nu) $. By \eqref{scheme_upper_bound1} the mapping $[0,T]\ni t \mapsto \sup\limits_{0 \leq u \leq t}\mathbb{E}\|\tilde{X}_{M,n}^{RE}(u)\|^p$ is bounded and Borel (as a non-decreasing function). Thus, \eqref{csa_3} together with the Gronwall's lemma imply \eqref{time_continuous_scheme_estimate}.  
\end{proof}

\end{document}